\documentclass[12pt]{article}
\usepackage{graphicx}
\usepackage{amssymb}
\usepackage{amsmath}
\usepackage{amsfonts}
\usepackage{amsthm}
\usepackage[english]{babel}
\usepackage[latin1,ansinew]{inputenc}
\usepackage{a4wide}
\usepackage{color}
\usepackage{academicons}
\newcommand{\be}{\begin{eqnarray}}
\newcommand{\ee}{\end{eqnarray}}

\swapnumbers
\newtheorem{theo}{Theorem}[section]
\newtheorem{remark}[theo]{Remark}
\newtheorem{lemma}[theo]{Lemma}
\newtheorem{coro}[theo]{Corollary}
\newtheorem{defi}[theo]{Definition}
\newtheorem{prop}[theo]{Proposition}
\newcommand{\R}{\mathbb R}
\newcommand{\C}{\mathbb C}
\newcommand{\N}{\mathbb N}
\newcommand{\mG}{{\mathcal G}}
\newcommand{\mD}{{\mathcal D}}
\newcommand{\mB}{{\mathcal B}}
\newcommand{\mL}{{\mathcal L}}
\newcommand{\mZ}{{\mathcal Z}}

\newcommand{\mM}{{\mathcal M}}
\newcommand{\mK}{{\mathcal K}}
\newcommand{\mJ}{{\mathcal J}}
\newcommand{\mH}{{\mathcal H}}
\newcommand{\mT}{{\mathcal T}}
\newcommand{\mU}{{\mathcal U}}
\newcommand{\mW}{{\mathcal W}}
\newcommand{\mE}{{\mathcal E}}
\newcommand{\mS}{{\mathcal S}}
\newcommand{\mF}{{\mathcal F}}

\newcommand{\bxi}{\boldsymbol\xi}

\newcommand{\bomega}{\boldsymbol\omega}

\newcommand{\lxi}{{\langle \xi \rangle}}

\newcommand{\eps}{\epsilon}

\newcommand{\op}{\operatorname{op}}
\newcommand{\Op}{\operatorname{Op}}
\newcommand{\Rep}{\operatorname{Re}}
\newcommand{\supp}{\operatorname{supp}}

\newcommand{\beq}{\begin{equation}}
\newcommand{\eeq}{\end{equation}}

\numberwithin{equation}{section}

\begin{document}

\title{Global existence and decay of small solutions for quasi-linear second-order uniformly dissipative hyperbolic-hyperbolic systems}
\author{Matthias Sroczinski\thanks{Department of Mathematics, University of Konstanz, 
		78457 Konstanz, Germany. matthias.sroczinski@uni-konstanz.de,   https://orcid.org/0000-0002-5472-2741}}

\maketitle
\begin{abstract}

 This paper is concerned with quasilinear systems of partial differential equations consisting of two hyperbolic operators interacting dissipatively. Its main theorem establishes global-in-time existence and asymptotic stability of strong solutions to the Cauchy problem close to homogeneous reference states. Notably, the operators are not required to be symmetric hyperbolic, instead merely the existence of symbolic symmetrizers is assumed. The dissipation is characterized by conditions equivalent to the uniform decay of all Fourier modes at the reference state. On a technical level, the theory developed herein uses para-differential operators as its main tool. Apparently being the first to apply such operators in the context of global-in-time existence for quasi-linear hyperbolic systems, the present work contains new results in the field of para-differential calculus. In the context of theoretical physics, the theorem applies to recent formulations for the relativistic dynamics of viscous, heat-conductive fluids notably such as that of Bemfica, Disconzi and Noronha \cite{bem18} (\textit{Phys. Rev. D,} 98:104064, 2018.).\\


\end{abstract}

\section{Introduction and main result}
In this paper, we study systems of partial differential equations that are given by the superposition of two hyperbolic operators and show that homogeneous states are nonlinearly stable in the sense that small perturbations thereof lead to global-in-time decaying solutions.
Concretely, we consider the Cauchy problem for quasi-linear systems of the form
\begin{align}\label{shhs1}
	\sum_{j=0}^d A^j(u(t,x))u_{x_j}(t,x)&=\sum_{j,k=0}^d (B^{jk}(u(t,x))u_{x_j}(t,x))_{x_k},\quad x_0=t \ge 0,~x=(x_1,\ldots,x_d)\in \R^d,\\
	\label{shhs2}
	u(0,x)&=u_0(x), \quad u_t(0,x)=u_1(x),\quad x \in \R^d,
\end{align}
where both the operator on the right hand side and the operator on the left hand side are hyperbolic and each of them acts dissipatively on the trajectories generated by the other one. Such systems occur in theroretical physics as recent formulations for the (special-)relativistic dynamics of viscous, heat conductive fluids \cite{ft14, ft16,  ft18, bem18, Fre20, bem21}. Our results apply to these formulations. The main theorem is the following.

\begin{theo}
	\label{theo:main}
	Consider $d \ge 3$, $s>d/2+1$, $\bar u \in \R^n$ and  let \eqref{shhs1} satisfy conditions (H$_A$), (H$_B$) and (D) from Section 3. Then there exist constants $\delta>0$ and $C=C(\delta)>0$ such that the following holds:
	For all $u_0,u_1$ with $u_0-\bar u  \in H^{s+1}(\R^d, \R^n)\cap L^1(\R^d,\R^n)$, $u_1 \in H^{s}(\R^d,\R^n)\cap L^1(\R^d,\R^n)$ as well as $\|u_0-\bar u\|_{H^{s+1}},\|u_1\|_{H^s},\|u-\bar u\|_{L^1}, \|u_1\|_{L^1} <\delta$ there exists a unique global solution $u$ of \eqref{shhs1}, \eqref{shhs2} satisfying $u-\bar u \in C([0,\infty),H^{s+1})\cap C^1([0,\infty),H^{s})$, $l=0,\ldots,s+1$ and, for all $t \in [0,\infty)$, 
	\begin{align}
		\label{eq:decay}
		\|u(t)-\bar u\|_{H^s}+\|u_t(t)\|_{H^{s-1}}  &\le C(1+t)^{-\frac{d}{4}}(\|u_0-\bar u\|_{H^s}+\|u_1\|_{H^{s-1}}+\|u_0-\bar u\|_{L^1}+\|u_1\|_{L^1}),\\
		\label{eq:energy}
		\|u(t)-\bar u\|_{H^{s+1}}^2+\|u_t(t)\|_{H^s}^2&+ \int_0^t\|u(\tau)-\bar u\|_{H^{s+1}}^2+\|u_t(\tau)\|_{H^s}^2~d\tau\\
		\nonumber
		&\le  C(\|u_0-\bar u\|_{H^{s+1}}^2+\|u_1\|_{H^s}^2
		+\|u_0-\bar u\|_{L^1}^2+\|u_1\|_{L^1}^2)
	\end{align}
\end{theo} 

While conditions (H$_A$) and (H$_B$) specify the assumed hyperbolicity, condition (D), essentially obtained in \cite{FS}, characterizes the needed decay behaviour for the Fourier modes of the associated linearized system. 

Based on the famous Kawashima-Shizuta condition \cite{kawa83, kawa85}, analogous results are well-known for symmetric hyperbolic-parabolic systems and first-order hyperbolic systems with relaxation, cf. \cite{che94, boi97, yon01, han03, kawa04, yon04, bia07} among others.\footnote{Note that the often available reformulations of \eqref{shhs1} as first-order hyperbolic systems do typically not satisfy the assumptions of these works.} Regarding  dissipative second-order hyperbolic systems there are  fewer results available, cf. notably \cite{dha11, liu13, rac17} and references therein, all of those treat systems whose structure is different form the one we consider here. The most prominent example for equations satisfying condition (D) are probably damped wave equations with a non-linear convection term, which alternatively can be viewed as conservation laws with hyperbolic artificial  viscosity. In this case, (D) reduces to Whitham's famous sub-characteristic condition \cite{whi74} and various in-depth results on the asymptotic behaviour of solutions have been achieved in this context, cf. \cite{Ori06, ued07, ued11, has12,che14, kat17}. Closest related to the present work are \cite{sro18, srot19, sro20}, there a predecessor of Theorem \ref{theo:main} was shown for the systems proposed in \cite{ft16,ft18}.

The  theory developed in the present work requires novel techniques in the use of para-differential operators. Developed by Bony \cite{bon81} and Meyer \cite{mey81, mey812}, such operators have been used in the context of hyperbolic equations by G\'{e}rard and Rauch \cite{GR}, Taylor \cite{T91} and M\'{e}tivier \cite{met01}. However, quite different from these works, we will in particular need to precisely understand how the norms of para-differential operators depending on the functions inducing their symbols. 

The paper is organized as follows. In the crucial Section 2 general results on para-differential operators needed for the argumentation in Section 3 and 4 will be established. The present work apparently being the first that uses such operators to treat global-in-time solutions to quasi-linear hyperbolic systems, we prove corresponding new results on that dependence. The technical highlight in this regard will be a modified version of the strong G\r{a}rding inequality. In Section 3 we construct a para-differential operator which is specifically associated with the system's dissipativity. Section 4 is dedicated to the proof of Theorem \ref{theo:main}. The challenging part is the treatment of the highest derivatives. Here we have to use the sophisticated estimates of Section 2 and the construction of Section 3. Finally, Section 5 shows that models of equations of dissipative relativistic fluid dynamics satisfy the assumptions of Theorem \ref{theo:main}.

\section{Results on para-differential operators}

A tour through the theory of para-differential operators from scratch to fine properties, this section relies on Appendix C of Benzoni-Gavage and Serre \cite{BS}  and Section 9 of Hörmander \cite{LH}, however with strong attention to symbols induced by what later will be the solution to the PDE system considered. In its initial part interpolating between brevity and legibility, the section culminates in the aforementioned novel version of the strong G\r{a}rding inequality.

\subsection{Notation, definitions and basics}

For topological vector-spaces $V,W$ we write $\mathcal L(V,W)$ for the space of continuous linear operators form $V$ to $W$ (or $\mathcal L(V)$ if $W=V$). Throughout this section consider fixed  dimensions $n,d \in \N$ and let $m$ denote some real number. For $x,\xi \in \R^d$ we just write $x\xi$ for their Euclidian scalar product.

Let $E$ be a finite-dimensional $\C$-Banach space. We denote the $E$-valued Schwartz space by $\mS(\R^d,E)$, and by $\mS'(\R^d,E):=\mathcal L(\mS(\R^d),E)$ the space of continuous linear mappings from $\mS(\R^d)$ to $E$, i.e. the space of $E$-valued temperate distributions, both equipped with the standard locally convex topologies. For $f \in \mS(\R^d,E)$ the Fourier transform is 
$$(\mathcal{F}f)(\xi)=\hat{f}(\xi)=(2\pi)^{-d/2}\int_{\R^d} f(x)\rm{e}^{-ix\xi}dx$$
with inverse
$$(\mathcal{F}^{-1}\hat{f})(x)=(2\pi)^{-d/2}\int_{\R^d}\hat f(\xi)\rm{e}^{ix\xi}d\xi.$$
We write $\mF_1$ and $\mF_2$ for the Fourier transform with respect to the first and the second variable for functions $f \in \mS(\R^d\times \R^d,E)$, i.e.
\begin{align*}
	(\mF_1f)(\eta,y)&=\mF(f(\cdot,y))(\eta)=(2\pi)^{-d/2}\int_{\R^d} f(x,y)\rm{e}^{-ix \eta}dx,\\
	 (\mF_2f)(x,\xi)&=\mF(f(x,\cdot))(\xi)=
	(2\pi)^{-d/2}\int_{\R^d} f(x,y)\rm{e}^{-iy\xi}dy.
\end{align*}
As usual we extend $\mF,\mF_1,\mF_2$  to continuous operators on   $\mS'(\R^d,E)$, $\mS'(\R^d \times \R^d,E)$ and unitary operators on $L^2(\R^d,E),L^2(\R^d\times\R^d,E)$ also denoted by $\mF, \mF_1, \mF_2$. 

We will use $\lxi:=(1+|\xi|^2)^{\frac12}$, $\xi \in \R^d$,  $\Lambda^m:=\mF^{-1}\langle \cdot \rangle^m\mF$.
As usual 
$$H^m(\R^d,E):=\{u \in L^2(\R^d,E): \Lambda^m u \in L^2(\R^d,E)\},$$
are the $L^2$-based $E$-valued Sobolev spaces with norm
$$\|u\|_{H^m(\R^d,E)}:=\|\Lambda^m u\|_{L^2(\R^d,E)}.$$
If $E$ is a Hilbert space we consider the scalar product on $H^m(\R^d,E)$
$$\langle u,v\rangle_{H^m(\R^d,E)}:=\langle \Lambda^m u,\Lambda^m v\rangle_{L^2(\R^d,E)}.$$

We also use $L^\infty$-based Sobolev spaces
$$W^{k,\infty}(\R^d,E):=\{u \in L^\infty(\R^d,E): \partial_x^\alpha u \in L^\infty(\R^d,E), ~|\alpha| \le k\}$$
with norm 
$$\|u\|_{W^{k,\infty}(\R^d,E)}=\max_{|\alpha| \le k}\|\partial_x^\alpha u\|_{L^\infty(\R^d,E)}.$$

We often just write $H^m$, $\|u\|_m$, $\langle u,v \rangle_m$, $W^{k,\infty}$ instead of $H^m(\R^d,E)$, $\|u\|_{H^m(\R^d,E)}$, $\langle u,v \rangle_{H^m(\R^d,E)}$, $W^{k,\infty}(\R^d,E)$ if there is no concern for confusion, and $\|u\|$ for $\|u\|_0$.

For $A \in \C^{n \times n}$ we denote the adjoint matrix by $A^*=\bar A^t$ and for $T \in \mL(\mS(\R^d,\C^n))$ we write $T^*$ for the adjoint operator with respect to the $L^2(\R^d,\C^n)$ inner product. As usual we call $T$ self-adjoint if $T=T^*$ and positive (strictly positive) if $\langle Tf,f\rangle_0 \ge 0$ ($\langle Tf,f \rangle >0$), in which case we also write $T\ge0$ ($T>0$).

Next, we turn to the basic definitions concerning pseudo-differential operators which will be used in the present paper. We consider the following symbol classes.
\begin{defi}
	\begin{enumerate}
		\item[(i)]
	 $S^m:=S^m(\R^d,\C^{n \times n})$ is the set of all functions $a \in C^\infty(\R^d \times \R^d, \C^{n \times n})$ for which for any $\alpha,\beta \in \N_0$ there exists $C_{\alpha\beta}>0$ such that
	\begin{equation}
		\label{eq:semin}
		|\partial_x^\beta\partial_\xi^\alpha a(x,\xi)| \le C_{\alpha\beta}\lxi^{m-|\alpha|}.
	\end{equation}
	With semi-norms being the optimal constants in \eqref{eq:semin},
	$S^m$ is a Fr\'{e}chet space.
\item[(ii)]
 $S^m_{1,1}:=S^m_{1,1}(\R^d,\C^{n\times n})$ is the set of functions $a \in C^\infty(\R^d\times \R^d)$ for which for any $\alpha,\beta \in \N_0^d$ there exist $C_{\alpha\beta}>0$ such that 
\begin{equation}
	\label{sym2}
	|\partial_x^\beta \partial_\xi^\alpha a(x,\xi)| \le C_{\alpha\beta}\lxi^{m-|\alpha|+|\beta}
\end{equation}
for all $(x,\xi) \in \R^d\times \R^d$. With semi-norms being the optimal constants in \eqref{sym2}, $S^m_{1,1}$ is a Fr\'{e}chet space.
\item[(iii)]
For $a \in S^m_{1,1}$ the mapping $\op[a] \in \mL(\mS(\R^d,\C^n))$ defined by
	\begin{equation}
		\label{eq:pseuop}
		(\op[a]f)(x):=(2\pi)^{-\frac{d}{2}}\int {\rm e}^{ix\xi}a(x,\xi)\mF f(\xi)~d\xi.
	\end{equation}
	is called the pseudo-differential operator with symbol $a$. We also write $a:=\operatorname{Sym}[\op[a]]$.
\end{enumerate}
\end{defi}
As first shown in \cite{bou83,bou88} for $a \in S^{m}_{1,1}$ the operator $\op[a]$ extends to a bounded operator from $H^{l+m}$ to $H^l$ only if $\op[a]^*$ also has a symbol in $S_m^{1,1}$. But the operator norm of $\op[a]$ can in general not be controlled by semi-norms of $a$ uniformly over this subspace. As for our applications to dissipative hyperbolic systems it is essential that the norm of $\op[a]$ is small if the semi-norms of $a$ are small we have to make sure that the symbols occurring in the present work belong to the following smaller subspaces.
\begin{defi}
For $L \in (0,1]$,  $S^{m,L}_{1,1}$  is the subspace of all $a \in S^m_{1,1}$ such that $\mF_1a$ vanishes on $N_L:=\{(\eta,\xi) \in \R^d \times \R^d: |\eta+\xi| +1<L|\xi|\}$ in the sense of distributions, i.e. 
\begin{equation}
\label{etaxi}
a(\mF_1\phi)=0~~ \text{for all } \phi \in \mS(\R^d \times \R^d)~\text{with } \supp \phi \subset N_L.
\end{equation}
\end{defi}

\begin{prop}
	\label{prop:Sm11}
	Let $L \in (0,1]$. For all $l \in \R$ and $a \in S^{m,L}_{1,1}$ $\op[a]$ extends to a continuous operator form $H^{l+m}$ to $H^l$ and $\op$ is itself continuous from $S^{m,L}_{1,1}$ to $\mathcal L(H^{l+m}, H^l)$. 
\end{prop}

\begin{proof}
	Cf. \cite{LH}, Proposition 9.3.1.
\end{proof}

The symbols in Sections 2 and 3 will be induced by functions $(x,\xi) \mapsto F(u(x),\xi)$ where $F \in C^{\infty}(\R^n \times \R^d)$, $u \in W^{k,\infty}(\R^d,\R^n)$ for some $k \in \R$, i.e. they belong to the following symbol class.

\begin{defi}
	For any $k \in \N_0$ the
 set $\Gamma^m_k$ of symbols of order $m$ with regularity $k$ is  the set of functions $A: \R^d \times \R^d \mapsto \C^{n\times n}$ such that,
	\begin{enumerate}
		
		\vspace{-1.5\baselineskip}
	\item[(i)]
	for almost all $x \in \R^d$ the mapping $\xi \mapsto A(x,\xi)$ is in $C^\infty(\R^d,\C^{n\times n})$
	\item[(ii)]
	
	\vspace{-0.5\baselineskip}
	for any $\alpha \in \N_0^d$ and $\xi \in \R^d$ the mapping $x  \mapsto \partial_\xi^\alpha A(x,\xi)$ belongs
	to $W^{k,\infty}(\R^d,\C^{n \times n})$ and there exists $C_\alpha > 0$ not depending on $\xi$ such that
	\begin{equation}
		\label{eq:semingamma}
		\|\partial^\alpha_\xi A(\cdot,\xi)\|_{W^{k,\infty}} \le C_\alpha \lxi^{m-|\alpha|}.
	\end{equation}
With the semi-norms being the optimal constants in \eqref{eq:semingamma},
$\Gamma^m_k$ is a Fr\'{e}chet space.
\end{enumerate}
\end{defi}

Para-differential operators associated with symbols in $\Gamma^{m}_k$ are  defined as follows.

\begin{defi}
	\label{lem:chi}
	For $\eps=(\eps_1,\eps_2)$ with $0<\eps_1<\eps_2 <1$ we call a function $\chi\in C^\infty(\R^d\times \R^d)$ an admissible $\eps$-cut-off if $\chi$ is even with respect to each variable, $\chi(\R^d \times \R^d) \subset [0,1]$, 
	\begin{equation}
		\label{supchi}
		\chi(\eta,\xi)=\begin{cases}
		1,&|\eta| \le \eps_1|\xi|~\text{and}~|\xi| \ge 1\\
		0,& |\eta|\ge \eps_2\lxi~\text{or}~|\xi| \le \eps_2
		\end{cases}	
	\end{equation}
 for all $\eta, \xi \in \R^d$ and for all $\alpha,\beta \in \N^d$  there exists $C_{\alpha,\beta}>0$ such that for all $\xi,\eta \in \R^d$
$$
|\partial_\eta^\beta\partial_\xi^\alpha\chi(\eta,\xi)| \le C_{\alpha,\beta}\lxi^{-|\alpha|-|\beta|}.$$
\end{defi}
 
\begin{prop}
	\label{prop:paradiff}
	Let $\chi$ be an admissible $\eps$-cut-off. Set $K^\chi:=\mF_1^{-1}(\chi)$
	and consider the function $R^\chi:\Gamma_k^m \to C^\infty(\R^d\times \R^d)$ given by
	$$R^\chi(A):=K^\chi \ast_1 A,~~A \in \Gamma_k^m.$$
	Then $R^\chi$ defines a bounded linear operator from $\Gamma_k^m$ to $S^{m,1-\eps_2}_{1,1} \cap \Gamma_k^m$. Here
	$$(K^\chi \ast_1 A)(x,\xi)=\int_{\R^d} K^\chi(x-y,\xi) A(y,\xi)~dy.$$
\end{prop}

\begin{proof}
Apart from the aspect that $a$ is not only in $S^{m}_{1,1}$ but even in $S^{m,1-\eps_2}_{1,1}$ the proof can be found in \cite{BS}, Proposition C.16. But that aspect follows in a straightforward manner as $|\eta+\xi| +1\le (1-\eps_2)|\xi|$ implies $|\xi|-|\eta|+1 \le (1-\eps_2)|\xi|$ and thus $|\eta| \ge \eps_2 \lxi$ and $\chi$ vanishes for such $\eta,\xi$.
\end{proof}

\begin{defi}
	Let $\chi$ be an admissible  $\eps$-cut-off. For $A \in \Gamma_k^m$ the ($\chi$-)para-differential operator with symbol $A$ is defined by
	$$\Op_\chi[A]:=\op[R^\chi(A)].$$
	As $R^\chi \in \mL(\Gamma^m_k,S^{m,1-\eps_2}_{1,1})$, $\Op_\chi=\op \circ R^\chi$ defines a continuous linear operator from $\Gamma^m_k$ to $\mathcal L(H^{l+m},H^l)$ ($l \in \R$). In particular the  $\mathcal L(H^{l+m},H^l)$-norm of $\Op_\chi[A]$ can be estimated by a constant depending on $l,\chi$ and a finite sum of $\Gamma_m^k$-semi-norms of $A$.
\end{defi}

The following shows that, regarding its dependence on $\chi$, $\op_\chi[A]$ is determined by $A$ up to a lower order operator, if $k \ge 1$.

\begin{lemma}
\label{lem:error}
Let $\chi$ be an admissible $\eps$-cut-off and $k\ge 1$. Then the following holds:

\vspace{-\baselineskip}
\begin{enumerate}
	\item[(i)]
	The mapping $R^\chi-\operatorname{Id}$ is a continuous operator from $\Gamma_k^m$ to $\Gamma^{m-1}_{k-1}$. 
	\item[(ii)]
	If $\tilde{\chi}$ is an admissible $\tilde{\eps}$-cut-off, then $R^\chi-R^{\tilde{\chi}}$ is a continuous operator from $\Gamma^m_k$ to $S^{m-1,1-\tau}_{1,1} \cap \Gamma^{m-1}_{k-1}$ with $\tau=\max\{\eps_2,\tilde{\eps_2}\}$.
\end{enumerate}
\end{lemma}	

\begin{proof}
	Cf. \cite{BS}, Proposition C.13, Corollary C.5.
\end{proof}

We end this subsection by stating two additional results on para-differential operators for later usage. The proofs are contained in \cite{BS}, Appendix C. To simplify notation we fix an admissible $\eps$-cut-off $\chi$ and suppress the dependence of $R^\chi$ and $\Op_\chi$ on $\chi$ in the following. We call an operator $K$ infinitely smoothing if $K \in \mathcal{L}(H^s,H^l)$ for all $s,l \in \R$.
\begin{lemma}
	\label{lem:is2}
	Let $b \in S^m$ be constant with respect to the first variable. Then the following holds:
	\begin{enumerate} 
		
	\vspace{-\baselineskip}
	\item[(i)]
	$\op(b)-\Op[b]$ is infinitely smoothing.
	\item[(ii)] $\Op[b]=\Op[b^*]$
	 \item[(iii)]$\Op[Ab]=\Op[A]\mF^{-1}b\mF$ for any $A \in \Gamma^{\mu}_k$.
	\end{enumerate}
\end{lemma}

\begin{lemma}
	\label{lem:diffmult}
	For each $k >0$ there exists $C>0$ such that for all $f \in L^\infty \cap H^k, A \in W^{1,\infty} \cap H^k$
	\begin{align*}
		\|A-\Op[A]f\|_k &\le C(\|A\|_{H^k}\|f\|_{L^\infty}+\|A\|_{W^{1,\infty}}\|f\|_{H^{k-1}}).
	\end{align*}
\end{lemma}

\subsection{Adjoints and products}

For the argumentation in Section 3 it will be essential to control the norms of operators $\Op[A^*]-\Op[A]^*$, $\Op[BA]-\Op[B]\Op[A]$, $A \in \Gamma_1^m, B \in \Gamma_1^\mu, \mu \in \R$, in terms of the semi-norms of $A$ and $B$. While for $a \in S^{m,L}_{1,1}$, $b \in S^{m,L}_{1,1}$    there exist symbols $g \in S^{m}_{1,1}, h \in S^{m+\mu}_{1,1}$ such that $\op[a]^*=\op[g]$, $\op[b]\op[a]=\op[h]$ and that, provided $\partial_{x_j}a \in S^{m}_{1,1} (j=1,\ldots,d)$, $\op[a^*]-\op[a]^* \in \mL(H^{l+m-1}, H^l)$, $\op[b]\op[a]-\op[ba] \in \mL(H^{l+m+\mu-1}, H^l)$, $l \in \R$, it is not true in general that $g,h$ are again in some class $S^{m,L}_{1,1}$, $S^{m+\mu,L}_{1,1}$ which would allow to control their operator norms. However, for our purposes it is sufficient to consider symbols of the particular form $a=R(A), b=R(B)$ for $A \in \Gamma^{m}_1, B \in \Gamma^\mu_1$ and we will show that in this case the symbols of $\op[a]^*=\Op[A]^*$, $\op[b]\op[a]=\Op[B]\Op[A]$ are in fact again in $S^{m,L}_{1,1}$, $S^{m+\mu,L}_{1,1}$ for some $L \in (0,1]$.

As a first step note that for symbols in $\mS(\R^d\times \R^d,\C^{n\times n})$ there exist neat formulas for the symbols of adjoint and product operators, which also can be found in \cite{LH}.

\begin{lemma}
	\label{lem:adjS}
	If $a \in \mS(\R^d \times \R^d)$, then $\op[a]^*=\op[g]$ with $\mF_1g(\eta,\xi)=(\mF_1a(-\eta,\eta+\xi))^*$.
\end{lemma}

\begin{lemma}
	\label{rem:prod}
	If $a,b \in \mS(\R^d \times \R^d,\C^{n \times n})$, then $\op[b]\op[a]=\op[h]$ with
	$$\mF_1h(\eta,\xi)=\int_{\R^d}\mF_1b(\eta-\theta+\xi,\theta)\mF_1a(\theta-\xi,\xi)d\theta.$$
\end{lemma}

The significance of this result lies in the following observation.

\begin{lemma}
	\label{lem:approx}
	Let $A \in \Gamma_0^m$. Then there exists a sequence $(a_\nu)_{\nu \ge 1} \subset \mS(\R^d \times \R^d,\C^{n \times n})$ such that  $\op[a_\nu]u \to \Op[A]u$, $\nu \to \infty$ in $\mS(\R^d,\C^n)$ for all $u \in \mS(\R^d,\C^n)$. Furthermore for  all $\delta \in (\eps_2,1)$, $\eps_2$ being the constant of the $\eps$-cut-off, there exists $\nu_0>0$ such that $\supp \mF_1a_\nu \subset \{(\eta,\xi) \in \R^d \times \R^d:|\eta| \le \delta \lxi\}$ for all $\nu\ge \nu_0$.
\end{lemma}

\begin{proof}
	The first part of the statement is shown as in \cite{LHl}, proof of Theorem 18.1.8. However, we have to slightly modify the construction to also obtain the second part. Set $a:=R(A)$.
	Choose $\hat \phi \in \mS(\R^d)$ with $\supp \hat\phi \subset B_0(1)$, $\mF^{-1}\hat\phi(0)=1$ and define $\phi:=\mF^{-1}\hat{\phi}$, 
	$$a_\nu(x,\xi):=\phi(x/\nu)\phi(\xi/\nu)a(x,\xi),~~x,\xi \in \R^d.$$
	The asserted convergence then follows as ibid. 
	
	It remains to show the statement concerning the supports. Set $\psi_\nu(x,\xi):=\phi(x/\nu)\phi(\xi/\nu)$ ($\xi,\eta \in \R^d, \nu \ge 1)$. As $\mF_1(a_\nu)=(2\pi)^{-d/2}(\mF_1\psi_\nu)\ast_1(\mF_1a)$
	and $\mF_1a=\chi\mF_1 A$, it is sufficient to show that for given $\delta \in (\eps_2,1)$ and $\nu$ sufficiently large  $\chi(\eta-\theta,\xi)\mF_1\psi_\nu(\theta,\xi)=0$ for all $\theta,\eta,\xi \in \R^d$ $|\eta| \ge \delta \lxi$. Clearly
	$$\mF_1 \psi_\nu(\theta,\xi)=\nu^d\hat \phi(\theta \nu)\phi(\xi/\nu).$$
	As by construction $\hat \phi(\theta \nu)=0$ for $|\theta| \ge \nu^{-1}$ we can assume $|\theta| \le \nu^{-1}$. Then $|\eta| \ge \delta \lxi$ yields
	$$|\eta-\theta| \ge |\eta| -|\theta| \ge \delta \lxi-\nu^{-1}.$$
	Hence choosing $\nu$ so large that $\nu^{-1} \le \delta-\eps_2$ gives (note $\lxi \ge 1$)
	$$|\eta-\theta| \ge \delta \lxi - (\delta-\eps_2) \lxi=\eps_2 \lxi.$$
	But this implies $\chi(\eta-\theta,\xi)=0$, which finishes the proof.
\end{proof}

\begin{prop}
	\label{prop:adj}
	Let $A \in \Gamma^m_0$. Then there exists $b=b(A) \in S^{m,1-\eps_2}_{1,1}$ such that $\Op[A]^*=\op[b(A)]$. Furthermore if $A \in \Gamma_1$ the operator
	$$T: \Gamma^m_1 \to  S^{m-1,1-\eps_2}_{1,1}, a \mapsto b(A)-R(A)^*$$
	is continuous. In particular the mapping
	$$ A \mapsto \Op[A^*]-\Op[A]^*=\op[R(A)^*-b(A)]$$
	is continuous from $\Gamma^m_1$ to $\mathcal L(H^{l+m-1},H^l)$ for any $l \in \R$.
\end{prop}

\begin{proof}
	Set $a:=R(A)$.  As $A \in S^{m,1-\eps_2}_{1,1}$, the existence of $b:=b(A) \in S^m_{1,1}$ with $\operatorname{op}[b]=\operatorname{Op}[A]^*$ follows by \cite{LH}, Lemma 9.4.1. 
	
	Next we prove that $\mF_1 b$ vanishes on $\mathcal N_{1-\eps_2}$. If $A \in \mS(\R^d\times \R^d,\C^{n \times n})$ also $a \in \mS(\R^d \times \R^d, \C^{n \times n})$ and Lemma \ref{lem:adjS} gives
	$$\mF_1b(\eta,\xi)=(\mF_1a(-\eta,\eta+\xi))^*.$$
	If now $|\eta+\xi|+1 \le (1-\eps_2)|\xi|$ then $\eps_2|\xi| \le |\eta|$ and thus
	$$\eps_2\langle \eta+\xi \rangle \le \eps_2(1+|\eta+\xi|) \le(1-\eps_2)\eps_2|\xi| \le (1-\eps_2)|\eta| \le |\eta|,$$
	which implies $\mF_1a(-\eta,\eta+\xi)=(\chi\mF_1A)(-\eta,\eta+\xi)=0$.
	
	For general $A$ choose a sequence $(a_\nu)_{\nu \ge 1} \subset \mS(\R^d \times \R^d,\C^{n \times n})$ with $\op[a_\nu]u \to \Op[A]u$ in $\mS(\R^d, \C^n)$ for all $u \in \mS(\R^d,\C^n)$.  This implies $\op[a_\nu]^*\to \Op[a]^*=\op[b]$ in $\mS'(\R^d\times \R^d,\C^{n \times n})$ and it is straightforward to show that this yields $b_\nu \to b \in \mS'(\R^d \times \R^d,\C^{n \times n})$, where $\mF_1b_\nu(\eta,\xi)=\mF_1a_\nu(-\eta,\eta+\xi)$. By Lemma \ref{lem:approx}  $\mF_1a_\nu(\eta,\xi)$ vanishes for $|\eta| \ge \delta\lxi$, if $\delta \in (\eps_2,1)$. As seen above this yields $b_\nu \in S^{m,1-\delta}_{1,1}$. In conclusion $b=\lim_{\nu \to \infty} b_\nu \in S^{m,1-\delta}_{1,1}$ for all $\delta>\eps_2$, i.e. $b \in S^{m,1-\eps_2}_{1,1}$.
	
	Lastly $A \in \Gamma_1^m$ directly gives $\partial_{x}^\delta A \in \Gamma_0^m$ and hence $\partial_x^\delta R(A)=R(\partial_x^\delta A) \in S^{m}_{1,1}$ ($|\delta|=1$) . By \cite{LH}, Lemma 9.6.1 (applied to $N=1$, $m_N=m-1$) we now obtain $b-R(A) \in S^{m-1}_{1,1}$ and its $S^{m}_{1,1}$-semi-norms are bounded by a constant times a sum of finitely many $S^{m}_{1,1}$-semi-norms of $\partial_x^\delta R(A)$ ($|\delta|=1$). As also $b-R(A) \in S^{m-1,1-\eps_2}$, the assertion follows by the continuity of $R$ and $\op$.
\end{proof}

Concerning the analysis of product operators we first consider the difference $R^\chi(AB)-R^\chi(A)R^\chi(B)$.

\begin{lemma}
	\label{lem:cutoff}
	For $A \in \Gamma^m_1$, $B \in \Gamma^\mu_1$ and an $\eps$-cut-off $\chi$ with $\eps_2<1/2$ we have $R^{\chi}(B)R^\chi(A) \in S^{m+\mu,1-2\eps_2}_{1,1}$. Furthermore the bilinear operator
	$$T: \Gamma^m_1 \times \Gamma^m_1 \to S^{m+\mu-1,1-2\eps_2}_{1,1},\quad (A,B) \mapsto R^\chi(AB)-R^\chi(A)R^\chi(B)$$
	is continuous.
\end{lemma}

\begin{proof}
	We suppress the superscript $\chi$ in the following. As $R(A) \in S^m_{1,1},R(B) \in S^{\mu}_{1,1}$, it is clear that $R(B)R(A) \in S^{m+\mu}_{1,1}$. Thus regarding the first assertion we need to show that $R(B)R(A)$ vanishes on $\mathcal N_{1-2\eps_2}$. Since $\mF_1(R(B)R(A))=(2\pi)^{-d/2}\mF_1R(B) \ast_1\mF_1R(B)$ and $\mF_1R(A)=\chi\mF_1A$, $\mF_1R(B)=\chi\mF_1B$, it is sufficient to prove that $\chi(\eta-\theta,\xi)\chi(\theta,\xi)$ vanish for all $\theta,\eta,\xi \in \R^d$ with $|\eta+\xi|+1 \le (1-2\eps_2)|\xi|$. Take such $\theta,\eta,\xi$. If $\chi(\theta,\xi) \neq 0$ then $|\theta| \le \eps_2\lxi$ and $|\eta+\xi|+1\le (1-2\eps_2)|\xi|$ implies $|\eta| \ge 2\eps_2|\xi|+1$. Together this yields
	$$|\eta-\theta| \ge |\eta|-|\theta| \ge 2\eps_2|\xi|+1-\eps_2\xi \ge \eps_2\lxi.$$
	Now $\chi(\eta-\theta,\xi)$ vanishes for such $\eta,\theta,\xi$, wich completes the argument.
	
	In regard to the continuity we write
	$$R(BA)-R(B)R(A)=R(BA)-BA+B(A-R(A))-(R(B)-B)R(A).$$
	Hence it follows from Lemma \ref{lem:error} (i) and the continuity of $R$ that $T$ is continuous as an operator to $\Gamma^{m+\mu-1}_0$. Thus the proof is finished if we show that each $S^{m+\mu-1}_{1,1}$-semi-norm can be bounded by a constant times a finite sum of $\Gamma^{m+\mu-1}_0$-semi-norms of $T(A,B)$. We show that even the following holds: For all $\alpha,\beta \in \N_0$ there exists $C_{\beta}>0$ such that
	\begin{equation} 
		\label{eq:TAB}
		|\partial_x^\beta \partial_\xi^\alpha T(A,B)(x,\xi)| \le C_{\alpha\beta} |\partial_\xi^\alpha T(A,B)(x,\xi)|\lxi^{|\beta|},~~ x, \xi \in \R^d.
	\end{equation}
	By Bernstein's Lemma applied to $\partial_\xi^\alpha T(a,b)(\cdot,\xi)$ (cf. e.g. \cite{BS}, Lemma C.3) this can be deduced from the fact that for all $\xi \in \R^d$
	$$
	\supp\big((\mF T(A,B))(\cdot,\xi)\big) \subset B(0,2\eps_2\lxi).
	$$
	 In fact,
	 $$\supp (\mF(R(ba)(\cdot,\xi)) \subset \supp(\chi(\cdot,\xi)) \subset B(0,2\eps_2\lxi)$$
	 holds by definition of $\chi$ and that $\mF_1(R(b)R(a))$ vanishes for all $\eta,\xi$ with $|\eta| >2\eps_2\lxi$ follows by the same argumentation as in the first part of the proof.
\end{proof}

We can now prove our main proposition concerning products of para-differential operators.

\begin{prop}
	\label{prop:product}
	Let $A \in \Gamma^m_0$, $B \in \Gamma^\mu_0$. Then for $L:=(1-\eps_2)^2$ there exists $h(B,A) \in S^{\mu+m,L}_{1,1}$ such that $\Op[B]\Op[A]=\op[h(B,A)]$. Furthermore the operator
	$$\Gamma^m_1\times \Gamma^m_1 \to\mathcal L(H^{l+\mu+m-1},H^l),\quad (B,A)\mapsto \Op_{\chi}[B]\Op[A]-\Op[BA]$$
	is continuous for all $l \in \R$.
\end{prop}

\begin{proof}
	The existence of a $h=h(B,A) \in S^{m+\mu}_{1,1}$ such that
	$$\op[h(B,A)]=\op[R(B)]\op[R(A)]=\Op[B]\Op[A]$$
	follows direclty from \cite{LH}, Lemma 9.5.1 as $R(A) \in S^{m,1-\eps_2}_{1,1}$. We now prove that $h$ satisfies \eqref{etaxi} for $L=(1-\eps_2)^2$. First assume $a:=R(A),b:=R(B) \in \mS(\R^d \times \R^d)$. By Lemma \ref{rem:prod} 
	\begin{equation} 
		\label{eq:prodsym}
		\mF_1h(\eta,\xi)=\int_{\R^d}\mF_1b(\eta-\theta+\xi,\theta)\mF_1a(\theta-\xi,\xi)d\theta.
	\end{equation}
	Let $\eta,\xi \in \R^d$ with $|\eta+\xi|+1 \le (1-\eps_2)^2|\xi|$. If $\mF_1a(\theta-\xi,\xi)=\mF_1R(A)(\theta-\xi,\xi) \neq 0$ we have $|\theta-\xi| \le \eps_2 \lxi \le \eps_2+\eps_2|\xi|$, which gives $(1-\eps_2)|\xi| \le |\theta|+\eps_2$. We arrive at
	\begin{align*}|\eta+\xi-\theta| &\ge |\theta|-|\eta+\xi| \ge |\theta|-(1-\eps_2)^2|\xi|+1 \ge |\theta|-(1-\eps_2)|\theta|-(1-\eps_2)\eps_2+1\\
		&=\eps_2 \theta+\eps_2+(1-\eps_2)^2\ge \eps_2 \langle \theta \rangle.
	\end{align*}
	But this implies $\mF_1b(\eta+\xi-\theta,\theta)=\mF_1R(B)(\eta+\xi-\theta,\theta)=0$, which finishes the argument.
	
	For general $A,B$ choose sequences $(a_\nu)_{\nu \ge 1}, (b_\nu)_{\nu \ge 1}, \subset \mS(\R^d \times \R^d)$ with $\op[a_\nu]u \to \Op[A]u$, $\op[b_\nu]u \to \Op[B]u$ in $\mS(\R^d \times \R^d,\C^{n})$ for all $u \in \mS(\R^d \times \R^d,\C^n)$ as constructed im Lemma \ref{lem:approx}. Then clearly 
	$$\op[h_\nu]u=\op[b_\nu]\op[a_\nu]u \to \Op[B]\Op[A]u=\op[h]u$$ 
	in $\mS(\R^d, \C^n)$, where $h_\nu$ is defined by \eqref{eq:prodsym} with $a,b$ replaced by $a_\nu, b_\nu$. This implies $h_\nu \to h$ in $\mS'(\R^d \times \R^d,\C^{n \times n})$. As for all $1>\delta>\eps_2$ $\supp \mF_1a_\nu, \supp \mF_1b_\nu \subset \{(\eta,\xi) \in \R^d \times \R^d:|\eta| \le \delta \lxi\}$ for $\nu$ sufficiently large we get by the same reasoning as above that for all $1>\delta>\eps_2$ $h_\nu$ vanishes on $\mathcal N_{(1-\delta)^2}$ for $\nu$ sufficiently large, which proves that $h$ vanishes on $\mathcal N_{(1-\eps_2)^2}$.
	
	To prove the second assertion note that by Lemma \ref{lem:error} (ii), the mapping $G \mapsto \Op_{\chi}[G]-\Op_{\tilde \chi}[G]$ is continuous from  $\Gamma^k_1$ to $\mathcal L(H^{l+k-1},H^l)$, $k,l \in \R$, for any admissible cut-offs $\chi,\tilde \chi$. Hence we can assume w.l.o.g $\eps_2<\frac12$. By Lemma \ref{lem:cutoff} and the continuity of $\op$
	$$(B,A) \mapsto \Op[BA]-\op[R(B)R(A)]=\op[R(BA)-R(B)R(A)]$$
	is also continuous as mapping from $\Gamma^m_1 \times \Gamma_1^m$ to $\mathcal L(H^{l+\mu+m-1},H^l)$. What is left to show ist the continuity of
	$$(B,A) \mapsto \Op[B]\Op[A]-\op[R(B)R(A)]=\op[h(B,A)-R(B)R(A)].$$
	As $R(A) \in S^{m,1-\eps_2}_{1,1}$ and 
	$$\partial_{x_j}R^{\tilde \chi}(A)=R(\partial_{x_j}A) \in S^{m}_{1,1},~~\partial_{x_j}R^{\tilde \chi}(B)=R(\partial_{x_j}B) \in S^{m}_{1,1},~~j=1,\ldots,d.$$ all semi-norms of  $h(B,A)-R(B)R(A)$  can be estimated by a constant times a finite sum of products of semi norms of $\partial_{x_j}R(A), \partial_{x_k}R(B)$.
	Thus as $h(B,A)-R(B)R(A) \in S^{m-1,L}_{1,1}$ for $l=\min\{1-2\eps_2,(1-\eps_2)^2\}$ the assertion follows from the continuity of $\op$ and $R$.
\end{proof}

\subsection{Estimates for operators with symbols induced by Sobolev functions}

In Section 3 the results of Sections 2.1, 2.2 are applied to symbols of the form $(x,\xi)\mapsto F(u(x),\xi)$, where $F \in C^\infty(\mU \times \R^d, \C^{n\times n})$ ($\mU \subset \R^n$ some $0$-neighbourhood) and $u \in H^s(\R^d,\R^n)$ for $s$ sufficiently large. For this purpose we prove the results below.

In the following let $\mU \subset \R^N$ be a $0$-neighbourhood.

\begin{defi}
		We denote by $S^m(\mU):=S^m(\mU,\C^{n \times n})$ the set of all functions $F \in C^\infty(\mU \times \R^d, \C^{n \times n})$ for which for any $\alpha,\beta \in \N_0^d$ there exists $C_{\alpha\beta}>0$ such that for all $(u,\xi) \in \mU \times \R^d$
	\begin{equation} 
		\label{eq:help*}
		|\partial_x^\beta\partial_\xi^\alpha F(u,\xi)| \le C_{\alpha\beta}\lxi^{m-|\alpha|}.
	\end{equation}
\end{defi}

For functions $F:\mU \times \R^d \to \C^{n\times n}$ and $u: \R^d \to \mU$ we consider the composition
	$$F_u: \R^d \times \R^d \to \C^{n \times n}, (x,\xi) \mapsto F(u(x),\xi).$$

\begin{lemma}
	\label{lem:base}
	Let $F \in S^m(\mU)$ and $u \in H^s$ with $s > d/2$. Then $F_u \in \Gamma^m_k$ for $k=[s-d/2]$ and for all $\alpha\in \N_0^d$ and each $\Gamma^m_k$-semi-norm $p_{\alpha}(F_u)$ it holds
	$$p_{\alpha}(F_u) \le C_\alpha(\|u\|_s,F),$$ and if additionally $F(0,\xi)=0$, then
	$$p_{\alpha}(F_u) \le \tilde{C}_\alpha(\|u\|_s,F)\|u\|_s,$$
	 where $C_\alpha, \tilde{C}_\alpha$ depend on $\alpha$, $F$ and continuously on $\|u\|_s$.
\end{lemma}

\begin{proof}
	By Sobolev embedding $H^{s} \hookrightarrow  W^{k,\infty}$. Thus we have $F_u(\cdot,\xi) \in W^{k,\infty}$ and 
	$$\|\partial_\xi^\alpha F_u(\cdot,\xi)\|_{W^{k,\infty}}\le C (\|u\|_{W^{k,\infty}})\|\partial_\xi^\alpha F(\cdot,\xi)\|_{W^{k,\infty}(\mU)}\le  C (\|u\|_{s})C_\alpha(F)\lxi^{m-|\alpha}.$$
  all $\xi \in \R^d$. If $F(0,\xi)=0$, we even get, for all $\xi \in \R^d$,
 $$\|\partial_\xi^\alpha F_u(\cdot,\xi)\|_{W^{k,\infty}} \le C(\|u\|_{W^{k,\infty}})\|u\|_{W^{k,\infty}}\|\partial_\xi^\alpha F_u(\cdot,\xi)\|_{W^{k,\infty}(\mU)}\le C(\|u\|_{s},F)\|u\|_s \lxi^{m-|\alpha|}.$$
\end{proof}

The following proposition will be central for the energy estimates in Section 3. It follows directly by the continuity of $\Op: \Gamma^m_k \to \mathcal L(H^{l+m},H^l)$ and Lemma \ref{lem:base} as well as Propositions \ref{prop:adj}, \ref{prop:product} and the facts that $\Op[F_0]^*=\op[F_0^*]$ and $\op[G_0F_0]-\op[G_0]\op[F_0]$ is infinitely smoothing by Lemma \ref{lem:is2}.

\begin{prop}
	\label{prop:central}
		Let $F \in S^m(\mU)$, $l \in \R$. Then  for all $u \in H^s$ with $s>d/2$ there exists $C_l=C_l(F,\|u\|)>0$ depending on $l,F$ and monotonically increasingly on $\|u\|_s$ such that:
		\begin{enumerate}
		\item[(i)]
		$\|\Op[F_u]\|_{\mathcal L(H^{l+m},H^l)} \le C_l(\|u\|_{s})$
		and for $F(0,\cdot)=0$
		$\|\Op[F_u]\|_{\mathcal L(H^{l+m},H^l)} \le C_l\|u\|_{s}$,
		\item[(ii)]
		for $s > d/2+1$, $\Op[F_u]^*-\Op[F_u^*] \in \mathcal L(H^{l-1+m},H^{m})$ and
		$$\|\Op[F_u]^*-\Op[F_u^*]\|_{\mathcal L(H^{l-1+m},H^{m})} \le C_l\|u\|_{s}$$
		\item[(iii)]
		for $G \in S^\mu(\mU)$ and $s>{d/2+1}$ there exist $C_{l,2}=C_{l,2}(G,\|u\|_s)$ depending on $G$ and monotonically increasingly on $\|u\|_s$ such that
		$$\|\Op[G_u]\Op[F_u]-\Op[G_uF_u]\|_{\mathcal L(H^{l+\mu-1+m},H^m)} \le C_{l,2}C_l\|u\|_{s}$$
		up to an infinitely smoothing operator, which is determined by $F(0,\cdot), G(0,\cdot)$.
	\end{enumerate}
\end{prop}

\begin{prop}
	\label{prop:dev}
		Let $F \in S^m(\mU)$ and $u \in C^1([0,T],H^s)$ ($T>0$) for $s>d/2$. Then for each $l \in \R$ the mapping
		$$[0,T] \to \mathcal L(H^{l+m},H^l),\quad t \mapsto \Op[F_{u(t)}]$$
		is continuously differentiable and there exists $C_l$ depending on $l$ and $F$ but not on $u$ such that for all $t \in [0,T]$
		\begin{equation}
			\label{propdev1}
			\|\frac{d}{dt}\Op[F_{u(t)}]\|_{\mathcal L(H^{l+m},H^l)} \le C_l\|\partial_tu(t)\|_{s_0}
		\end{equation}
\end{prop}
	
\begin{proof}
	If $\frac{d}{dt} F_{u(t)} \in \Gamma^m_0$ we get by continuity and linearity of $\Op$
	$$\frac{d}{dt}\Op[F_{u(t)}]=\Op[\partial_tF_{u(t)}]$$
	To prove this and \eqref{propdev1} it is sufficient to show that for any $\alpha \in \N_0^d$ there exists $C_\alpha=C_\alpha(F)$ auch that for all $\xi \in \R^d$
	$$\|\partial_\xi^\alpha\partial_tF_{u(t)}(\cdot,\xi)\|_{L^\infty} \le C_\alpha\|\partial_tu(t)\|_{s}\lxi^{m-|\alpha|}.$$
	Let $\alpha\in \N_0^d$  and set $F^{\alpha}_{u(t)}:= \partial_\xi^\alpha F_{u(t)}$. 
	We have for all $x,\xi \in \R^d$
	$$\partial_tF^\alpha_{u(t)}(x,\xi)=\sum_{j=1}^n\partial_t u^j\partial_{u^j}F^\alpha(u(t,x),\xi)$$
	Due to $F \in S^m(\mU)$ this yields 
	$$\|\partial_tF^\alpha_{u(t)}(x,\xi)\|_{L^\infty}\le \|\partial_tu(t)\|_{L^\infty}\sum_{|\beta|=1}\|\partial_u^\beta F^\alpha(\cdot,\xi)\|_{L^\infty} \le C_\alpha(F)\|\partial_t u\|_{s}\lxi^{m-|\alpha|}.$$
\end{proof}

Lastly we prove a version of the strict G\r{a}rding inequality for $F \in S^m(\mU)$. First consider the following lemma which is a modification of a construction in \cite{LHl}, proof of Thm. 18.1.6.

\begin{lemma}
	\label{lem:garding}
	There exists an even function $\psi \in \mS(\R^d \times \R^d)$ with unit integral, $\Op[\psi]=\Op[\psi]^*$,  $\langle \op[\psi]v,v\rangle \ge 0$ ($v \in \mS(\R^d))$ and $\mF_1 \psi$ compactly supported.
\end{lemma}

\begin{proof}
	Choose an even function $\hat{\phi} \in C_0^\infty(\R^d\times \R^d)$ with $L^2$-norm one and set $\phi=\mF_1^{-1}\hat{\phi}$. By definition $\mF_1\phi$ is compactly supported and clearly $\phi$ is even and has $L^2$-norm one. Next, let $\psi \in \mS(\R^d)$ be the symbol of $\op[\psi]^*\op[\psi]$. As ibid. it follows that $\psi$ is even and has unit integral. $\op[\psi]=\op[\psi]^*$, $\langle \op[\psi]v,v\rangle_{L^2} \ge 0$ ($u \in \mS(\R^d))$ holds by definition. Now, let $\rho$ be the symbol of $\op[\phi]^*$. By Lemma \ref{lem:adjS} we get
	$$\mF_1 \rho(\eta,\xi)=(\mF_1 \phi)^*(-\eta,\eta+\xi),\quad \eta,\xi \in \R^d$$
	and thus by Lemma \ref{rem:prod}
	$$\mF_1\psi(\eta,\xi)=\int_{\R^d}\mF_1\rho(\eta-\theta,\theta+\xi)\mF_1\phi(\theta,\xi)d\theta=\int_{\R^d}\mF_1\phi(\theta-\eta,\eta+\xi)\mF_1\phi(\theta,\xi)d\theta.$$
	As $\mF_1\phi$ is compactly supported, we can choose $C>0$ such that $\mF_1\phi(\theta,\xi)=0$ if $|\theta| \ge C$ or $|\xi| \ge C$. Then by definition $\mF_1\psi(\eta,\xi)=0$ if $|\xi| \ge C$. Given $|\eta| \ge 2C$ and $|\theta| \le C$ we conclude $|\theta-\eta| \ge |\eta|-|\theta| \ge C$, i.e. $\mF_1\phi(\theta-\eta,\eta+\xi)=0$. In conclusion we have proven that $\mF_1\psi$ is in fact compactly supported. In particular $\psi \in \mS(\R^d \times \R^d)$. 
\end{proof}

Next we introduce a method to decompose symbols in $S^{m}_{1,1}$ into an infinite sum of infinitely smoothing symbols; cf. \cite{LH}. 

First, choose a function $\rho \in \mathcal{D}(\R^d)$ even and monotonically decaying along rays such that $\rho(\R^d) \subset [0,1]$ and
$$	\rho(\xi)=\begin{cases}
	1, & |\xi| \le \frac12\\
	0, & |\xi| \ge 1\\
\end{cases}.$$
For $\nu \in \N_0$ define $\rho_\nu, \zeta_\nu \in \mathcal{D}(\R^d)$ by 
$$\rho_\nu(\xi):=\rho(\xi/2^{\nu}),\quad \zeta_{\nu}(\xi)=\rho_{\nu+1}(\xi)-\rho_{\nu}(\xi),~\xi \in \R^d$$
Additionally set $\zeta_{-1}:=\rho$.

\begin{defi}
	\label{def:lp}
	For a function $a:\R^d\times \R^d \to \C^{n \times n}$ and $\nu \ge -1$ define
	$$a_{\nu}(x,\xi):=a(x,\xi)\zeta_\nu(\xi).$$
	Note that $a=\sum_{\nu \ge -1} a_\nu$.
\end{defi}

It is straightforward to show the following.

\begin{lemma} 
	\label{lem:lp}
	Let $a \in S^{m}_{1,1}$. Then  $a_\nu \in S^{-r}$ for all $r \in \R$ and for any $\alpha,\beta \in \N_0$ $x,\xi \in \R^d$
	$$|\partial_x^\beta\partial_\xi^\alpha a_\nu(x,\xi)|\lxi^r \le C  2^{\nu(r+m-|\alpha|+|\beta|)} \sum_{\gamma \le \alpha} C_{\gamma\beta}(a),$$
	where $C_{\gamma\beta}(a)$ are semi-norms of $a$.
\end{lemma}

\begin{prop}
	\label{prop:gading}
	 Let $s>d/2$, $u \in H^{s+2}$ and $F \in S^m(\mU)$ such that there exists an $R>0$ with $F(y,\xi)+F(y,\xi)^* \ge 0$ for all $y \in \mU$ and $\xi \in \R^d$ with $|\xi|>R$. Then there exists $C=C(\|u\|_{s+2},F)>0$ and for all $q \in \R$ there exists $c=c(\|u\|_{s+2},F,q)>0$, both increasing functions of $\|u\|_{s+2}$, such that for all $v \in \mS(\R^d,\C^n)$
	 $$\langle (\Op[F_u]+\Op[F_u]^*)v,v\rangle_{L^2} \ge -C\|u\|_{s+2}^{\frac12}\|v\|^2_{(m-1)/2}-c\|v\|_{-q}^2.$$
\end{prop}

\begin{proof}
	In the following it is straightforward to see that all constants can be chosen to be increasing functions of $\|u\|_{s+2}$. First note that by Proposition \ref{prop:central} for all $l \in \R$
	$$\|\Op[F_u]+\Op[F_u]^*-\Op[F_u+F_u^*]\|_{\mathcal L(H^{l+m-1,H^l})} \le C_{l} \|u\|_{s+1}.$$
	Thus 
	$$\langle (\Op[F_u]+\Op[F_u]^*)v,v\rangle_{L^2} \ge \langle \Op[F_u+F_u^*]v,v\rangle_{L^2}-C\|u\|_{s+1}\|v\|_{(m-1)/2}^2,~~v \in \mS(\R^d).$$
	Hence it is sufficent to prove the result for $\Op[F_u]+\Op[F_u]^*$ replaced by $\Op[F_u+F_u^*]$, i.e. we can assume w.l.o.g $F(u,\xi)=F(u,\xi)^* \ge 0$.
	
	It holds $R(F_u)=R(F_u^*)=R(F_u)^*$. By assumption this gives pointwise in $\R^d \times \{|\xi| \ge R\}$ for all $v \in \C^n$
	\begin{align*}
		\langle (R(F_u))v,v\rangle_{\C^n} &\ge \langle (R(F_u)-F_u)v,v \rangle_{\C^n} \ge -|R(F_u)-F_u||v|^2\\
		& \ge -(|R(F_u-F_0)-(F_u-F_0)|+|R(F_0)-F_0|) |v|^2.
	\end{align*}
	By Lemma \ref{lem:base} $F_u-F_0 \in \Gamma^m_2$ with all semi-norms bounded by a positive constant depending on $F$ times $\|u\|_{s+2}$. By Lemma
	 \ref{lem:error} (i) this yields $R(F_u-F_0)-(F_u-F_0) \in \Gamma^{m-1}_{1}$ with semi-norms bounded in the same way. Thus
	$$|R(F_u-F_0)-(F_u-F_0)| \le C_0\|u\|_{s+2}\lxi^{m-1}.$$
	Using also that $R(F_0)-F_0$ has compact support we conclude that for all $q \in \R$
	$$|R(F_u-F_0)-(F_u-F_0)|+|R(F_0)-F_0| \le C_0\|u\|_{s+2}\lxi^{m-1}+c_q^0\lxi^{-q}.$$
	Therefore on $\R^d \times \{|\xi| \ge R\}$
	$$a:=R(F_u)+C_0\|u\|_{s+2}\lxi^{m-1}+c_0\lxi^{-r} \ge 0$$
	and $a=a^*$, $a \in S^{m,1-\eps_2}_{1,1}$. As
	$$\Op[F_u]=\op[R(F_u)]=\op[a]-C_0\|u\|_{s+2}\op[\lxi^{m-1}]-c_0\op[\lxi^{-r}]$$
	it is now sufficient to show 
	$$\langle \op[a]v,v\rangle_{L^2} \ge -C\|u\|_{s+2}^{1/2}\|v\|^2_{(m-1)/2}-c\|v\|_{-q}$$
	for all $q \in \R$.
	
	To this end we proceed similarly as in the proof of Theorem 9.7.1 in \cite{LH} but with a crucial modification. First, decompose $a=\sum_{\nu \ge -1} a_\nu$ according to Definition \ref{def:lp}. As for all $\nu_0$ $\bar a_{\nu_0}:=\sum_{\nu=-1}^{\nu_0} a_{\nu} \in S^{-q}$ for any $q \in \R$ with norm depending on $\mu,\nu_0$ according to Lemma \ref{lem:lp}, i.e. $\|\op[\bar a_{\nu_0}]v\| \le c_{\nu_0,\mu}\|v\|_{-r}$, we only need to consider $\sum_{\nu \ge \nu_0} a_\nu$ for some $\nu_0 \in \N$. Naturally, in a first step we choose $\nu_0$ large enough to obtain $2^{\nu_0-2}>R$ and thus by assumption $a_\nu(x,\xi) \ge 0$ for all $x,\xi \in \R^d$, $\nu \ge \nu_0$. But we will later see that we may have to choose $\nu_0$ even larger.
	
	W.l.o.g. assume $u \neq 0$. Otherwise the result readily follows as $F_0 \ge 0$ is constant with respect to $x$ and $\Op[F_0]-\op[F_0]$ is infinitely smoothing.
	
	Choose an even function $\psi \in \mS(\R^d \times \R^d)$ with unit integral such that  $\op[\psi]=\op[\psi]^*$, $\langle \op[\psi]v,v\rangle \ge 0$ ($v \in \mS(\R^d))$ and $\mF_1 \psi$ compactly supported as constructed in Lemma \ref{lem:garding}. For $\nu \in \N_0$ set $q_\nu:=2^{\nu/2}$ and write $a_\nu=b_\nu+h_\nu$ with 
	\begin{align} 
		\label{eq:gard1}
		b_{\nu}(x,\xi)&:=\int_{\R^d}\int_{\R^d}\psi((x-y)q_\nu\mu,(\xi-\theta)/(q_\nu\mu))a_\nu(y,\theta)~dy~d\theta\\
		\label{eq:gard2}
		&=\int_{\R^d}\int_{\R^d}\psi(y,\theta)a_\nu(x-y/(q_\nu\mu),\xi-\theta q_\nu\mu)~dy~d\theta,
	\end{align}
	where $\mu:=\|u\|_{s+2}$.  As $a_\nu \ge 0$ and $\op[\psi]$ is a positive operator it is straightforward to obtain the positivity of $b_\nu$. Hence the theorem is proven provided
\begin{equation}
	\label{eq:gard3}
	\langle \op[h]v,v \rangle_{L^2} \ge -C_\mu\|u\|_{k+1}^{\frac12}\|v\|_{(m-1)/2}, ~~v \in \mS(\R^d).
\end{equation}
To this end we show $h \in S^{m-1,L}_{1,1}$ for some $L \in (0,1)$ and that all semi-norms of $h$ are bounded by a constant times $\|u\|_{s+2}^{\frac12}$.
Then \eqref{eq:gard3} follows from Proposition \ref{prop:Sm11}.

First we verify $h \in S^{m-1}_{1,1}$ and the estimate on the semi-norms, i.e.
\begin{equation} 
	\label{eq:c}
	|\sum_{\nu \ge \nu_0} \partial_x^\beta\partial_\xi^\alpha h_\nu(x,\xi)| \le C_{\alpha\beta}\|u\|_{s+2}^{\frac12} \langle \xi \rangle^{m-1-|\alpha|+|\beta|}.
\end{equation}
Let $\alpha=\beta =0$. Fix $\xi \in \R^d$ and consider $\nu \in \N_0$ with $|\xi|<2^{\nu-2}$ or $|\xi|>2^{\nu+2}$. As $a_\nu(y,\theta)=0$ for $2^{\nu-1} \le |\theta| \le 2^{\nu+1}$ we then have $h_\nu(x,\xi)=-b_\nu(x,\xi)$ and it follows by basic estimates (cf. \cite{LH}) that in the support of the first integrand in \eqref{eq:gard1}
$$|\xi-\theta|\ge \frac15(2^\nu+|\xi|)$$ 
and thus 
\begin{equation} 
	\label{eq:etamtheta}
	|\xi-\theta|/q_\nu=2^{-\nu/2}|\xi-\theta|\ge \frac15(2^\nu+|\xi|)^{\frac12}.
\end{equation}
As $a \in S^{m}_{1,1}$ and $\supp a_\nu \subset \{(x,\theta) \in \R^d\times \R^d: 2^{\nu-1} \le |\theta| \le 2^{\nu}\}$
$$|a_\nu(y,\theta)| \le C\langle \theta \rangle^m \le C(1+2^\nu)^{m}.$$
Hence $\psi \in \mS(\R^d \times \R^d)$ and \eqref{eq:etamtheta} yield 
\begin{equation}
	\label{eq:bnu}
	\begin{split}
	|h_\nu| &\le C_m(1+2^\nu)^{m} \\
	&\quad\int \int (|\xi-\theta|/(q_\nu\mu))^{-2(|m|+1)}(1+|\xi-\theta|/(q_\nu\mu))^{
	-n-1}(1+|(x-y)|q_\nu\mu)^{-n-1}dy~d\theta\\
	&\le C_{m,n}\mu^{2(|m|+1)}(1+2^\nu)^{m}(2^\nu+|\xi|)^{-2|m|-2}\\
	&\le C_{m,n}\mu^{2(|m|+1)}(1+|\xi|)^{m-1}2^{-\nu}.
	\end{split}
\end{equation}
Thus
\begin{equation} 
	\label{eq:sumeta}
	\sum_{\{\nu: |\xi|<2^{\nu-2} ~\text{or}~|\xi|>2^{\nu+2} \}} |h_\nu| \le C\|u\|_{s+2}^{\frac12}\langle \xi \rangle^{m-1}.
\end{equation}
Now consider $\nu \in \N_0$ with $2^{\nu-2} \le |\xi| \le 2^{\nu+2}$. As $\psi$ is an even function with unit integral we get from \eqref{eq:gard2}
\begin{equation}
	\label{eq:abnu1} 
	\begin{split}
	h_\nu&=a_\nu-b_\nu=\int \int \psi(y,\theta)\big(a_\nu(x,\xi)-a_\nu(x-y/(q_\nu\mu),\xi-\theta q_\nu\mu)\big)~dy~d\theta\\
	&=\int \int \psi(y,\theta)\big(\sum_{|\alpha+\beta|<2}\partial_x^\beta\partial_\xi^\alpha a_\nu(x,\xi)(-y)^\beta(-\theta)^\alpha-a_\nu(x-y/(q_\nu\mu),\xi-\theta q_\nu\mu)\big)~dy~d\theta.
	\end{split}
\end{equation}
By Taylor's fomula we can estimate (w.l.o.g. assume $|\theta| \le |\xi|$)
\begin{equation}
	\label{eq:abnu2} 
	\begin{split}
	&\quad\big|\sum_{|\alpha+|\beta|<2}\partial_x^\beta\partial_\xi^\alpha a_\nu(x,\xi)(-y)^\beta(-\theta)^\alpha-a_\nu(x-y/(q_\nu\mu),\xi-\theta q_\nu\mu)\big|\\
	&\le C\sum_{|\alpha|+|\beta|=2} \sup_{x,\xi \in \R^d} |\partial_x^\beta \partial_\xi^\alpha a_\nu(x,\xi)||y^{\beta}\theta^\alpha|(q_\nu\mu)^{|\alpha|-|\beta|}.
	\end{split}
\end{equation}

Note that $a=R(F_u)$. By Lemma \ref{lem:base} $F_u \in \Gamma^{m}_{2}$ and thus $\partial_x^\beta F_u \in \Gamma^{m}_{2-|\beta|}$ for $|\beta| \le 2$.  Hence  for each $\gamma \in \N_0^d$, $\xi \in \R^d$
$$\|\partial_\xi^\gamma\partial_x^\beta F_u(\cdot,\xi)\|_{W^{2-|\beta|,\infty}} \le C_\gamma\lxi^{m-|\gamma|}$$
and for $|\beta| \ge 1$  we also have $\partial_x^\beta F_u|_{u=0}=0$. Thus again by Lemma \ref{lem:base} 
$$\|\partial_\xi^\gamma\partial_x^\beta F_u(\cdot,\xi)\|_{W^{2-|\beta|,\infty}}  \le C_\gamma\|u\|_{s+2}\lxi^{m-|\gamma|}.$$
Clearly
$$\partial_x^\beta a=\partial_x^\beta R(F_u)=R\big(\partial_x^\beta F_u).$$
and we conclude from Proposition \ref{prop:paradiff} that $\partial_x^\beta a \in S^{m}_{1,1}$ and for all $x,\xi \in \R^d$
$$
	|\partial_x^\beta\partial_\xi^\gamma a(x,\xi)| 
	\le C_\gamma \lxi^{m-|\gamma|}\begin{cases}
		1,& |\beta|=0,\\
		\|u\|_{s+2},& 1\le |\beta| \le 2
	\end{cases}.
$$
Then by Lemma \ref{lem:lp} 
\begin{equation}
	\label{eq:abnu3}
	\sup_{x,\xi \in \R^d}| \partial_x^\beta\partial_\xi^\gamma a_\nu| \le C_\gamma 2^{\nu(m-|\alpha|)}\le C_\gamma\begin{cases}
	1,& |\beta|=0,\\
	\|u\|_{s+2},& 1\le |\beta| \le 2
\end{cases}
\end{equation}
From \eqref{eq:abnu1}, \eqref{eq:abnu2}, \eqref{eq:abnu3} and $\mu=\|u\|_{s+2}^{\frac14}$, $q_\nu=2^{\nu/2}$ we now get for $2^{\nu-2} \le |\xi| \le 2^{\nu+2}$
\begin{align*}|h_\nu| & \le C \int \psi(y,\theta)(|\theta|^2+|\theta||y|+|y|^2)~dy~d\theta\\ &\quad \big(2^{\nu(m-2)}2^\nu\|u\|_{s+2}^\frac12+2^{\nu(m-1)}\|u\|_{s+2}+2^{\nu m}\|u\|_{s+2}2^{-\nu}\|u\|_{s+2}^{-\frac12}\big)\\
	&\le C_\mu2^{\nu(m-1)}\|u\|_{s+2}^\frac12 \le C\|u\|_{s+2}^\frac12\langle \xi \rangle^{m-1},
\end{align*}
where we used $\psi \in \mS(\R^d \times \R^d)$ and $2^{\nu-2} \le |\xi| \le 2^{\nu+2}$ in the last line. Together with \eqref{eq:sumeta} this shows \eqref{eq:c} for $\alpha=\beta=0$.

Now note that  $\partial_x^\beta\partial_\xi^\alpha b_\nu$ is given by \eqref{eq:gard2}  with $a_\nu$ replaced  by $\partial_x^\beta\partial_\xi^\alpha a_\nu$. Hence we obtain \eqref{eq:c} for $\alpha,\beta \neq 0$ by applying the argumentation above with $a_\nu$ replaced by $\partial_x^\beta\partial_\xi^\alpha a_\nu$ and $m$ replaced by $m-|\alpha|+|\beta|$. 

To finish the proof we show that 
$\mF_1h$ vanishes on $\mathcal{N}_{L}=\{(\eta,\xi) \in \R^d\times \R^d:|\eta+\xi| <L|\xi|\}$ with $L:=\min\{1-\eps_2,\frac12\}$.
Then the estimate on the operator norm follows by the continuity of $\op$.
As $a=R(F_u) \in S^{m,1-\eps_2}_{1,1}$, it suffices to prove that $\mF_1b$ vanishes on $\mathcal{N}_{\frac12}$. 

By standard arguments on convolution and Fourier transform we have for all $g \in \mS(\R^d \times \R^d)$ 
$$b_\nu(\mF_1g)=(\mu q_\nu)^{-d/2}\int_{\R^d}\int_{\R^d}a_\nu(y,\theta)\mF_1f(y,\theta)~d\theta ~dy,$$
where
\begin{equation}
	\label{eq:fxi}
	f(\eta,\theta)=\int_{\R^d} \mF_1\psi(\eta/(q_\nu\mu),(\xi-\theta)/q_\nu\mu)g(\eta,\xi)d\eta.
\end{equation}
Let $\supp g \subset \mathcal N_{1/2}$. By construction we have $\supp \mF_1 \psi \subset \{(\xi,\eta) \in \R^d\times \R^d: |\eta|, |\xi| \le D,\}$ for some $D>0$. Next choose $\nu_0 \in \N$ so large that $3D\mu \le q_{\nu_0}/2$. Then for $\nu \ge \nu_0$ on the support of the integrand of \eqref{eq:fxi} we have $|\eta|,|\xi-\theta| \le Dq_\nu\mu$ and $|\xi+\eta|+1< \frac12|\xi|$. The first and third inequality yield
$$|\xi| < 2|\eta|\le 2Dq_\nu\mu$$ 
and thus the second one gives
$$|\theta| \le Dq_\nu\mu+|\xi| < 3Dq_\nu\mu\le  q_\nu q_{\nu_0}/2\le2^{\nu/2} 2^{\nu_0/2-1}\le2^{\nu-1}.$$
But this implies $b_\nu(y,\theta)=0$ for all $y \in \R^d$. Therefore we have proven $b_\nu(\mF_1g)=0$ for all $\nu\ge\nu_0$ and $\supp g \subset \{(\xi,\eta) \in \R^d \times \R^d: |\xi+\eta| < \frac12|\xi|\}$. Hence this also holds for $b=\sum_{\nu \ge \nu_0}b_\nu$.

\end{proof}

\section{Dissipativity}

Throughout this section we consider \eqref{shhs1}, \eqref{shhs2} with smooth matrix families $A^j,B^{jk}:\mU \to \R^{n \times n}$, $u_0,u_1: \R^d \to \mathcal U$ and $u:[0,T] \times \R^d \to \mathcal U$ for some domain $\mU \subset \R^n$. Carrying out the differentiation with respect to $x_k$ on the right-hand side and distinguishing between space and time derivatives we write \eqref{shhs1} as
$$-B^{00}(u)u_{tt}=\sum_{j,k=1}^d B^{jk}(u)u_{x_jx_k}+\sum_{j=1}^d (B^{0j}(u)+B^{j0}(u)u_{t})_{x_j}-A^0(u)u_t-\sum_{j=1}^dA^j(u)u_{x_j}+Q(u,D_{t,x}u),$$
where $Q$ is of the form
$$Q(u,D_{t,x}u)=\sum_{l=1}^n\sum_{j,k=0}^dQ^{ljk}(u)u_{x_k}^lu_{x_j}.$$
We will see in the proofs that the specific form of the matrices $Q^{ljk}(u)$ does not play any role. Hence  multiplying \eqref{shhs1} by $(-B^{00})^{-1}$, we can assume $-B^{00}=I_n$ without loss of generality, which we will always do in the following.

Next, denote by
\begin{align*}
	 B(u,\bxi)&:=\sum_{j,k=1}^d B^{jk}(u)\xi_j\xi_k,~~ C(u,\bxi):=\sum_{j=1}^d (B^{0j}(u)+B^{j0}(u))\xi_j,\\
	A(u,\bxi)&:=\sum_{j=1}^d A^j(u)\xi_j,~~\bxi=(\xi_1,\ldots,\xi_n) \in \R^d.
\end{align*}
the symbols of the second and first order parts, respectively. Then the hyperbolicity of both sides of \eqref{shhs1} is expressed by the following conditions:

\begin{itemize}
	\item[(H$_A$)] (a) there exists a smooth bounded family of hermitian uniformly  positive definite  matrices $\Sigma: \mU \to \R^{n \times n}$ such that $\Sigma(u)A^0(u)$ is symmetric and uniformly positive on $\mU$,\\
	(b) the matrix family $ A_0(u)^{-1}A(u,\bxi)$ permits a symbolic symmetrizer $H(u,\bxi)$,
	\item[(H$_B$)]  with
	$$
	{\mB(u,\bxi)} =
	\begin{pmatrix} 0 & |\bxi|I_n  \\
		-|\bxi|^{-1}B(u,\bxi) & i C(u,\bxi) \end{pmatrix}, 
	\quad \bxi=(\xi_1,...,\xi_d)\in\R^d,$$
	the matrix family $i\mathcal{B}(u,\bxi)$ permits a symbolic symmetrizer $\mathcal H(u,\bxi)$.
\end{itemize}

Above we use the following notion of a symbolic symmetrizer (cf. e.g. \cite{T91}).

\begin{defi}
	\label{def:symsym}
	Let $K \in C^\infty(\mU \times \R^d\setminus\{0\},\C^{n \times n})$. A symbolic symmetrizer for $K$ is a smooth mapping $S \in C^\infty(\mU \times \R^d\setminus\{0\},\C^{n \times n})$ positive homogeneous of degree $0$ with respect to the second argument, bounded as well as all its derivatives on $\mU \times \mathbb S^{d-1}$  such that for some $c>0$ and all $(u,\bxi) \in \mU \times \R^d\setminus\{0\}$
	$$S(u,\bxi)=S(u,\bxi)^* \ge cI_n,$$
	and $S(u,\bxi)K(u,\bxi)=(S(u,\bxi)K(u,\bxi))^*$. 
\end{defi}

\begin{remark}
	\label{rem:multis}
$K$ admits a symbolic symmetrizer if $K$ is positive homogeneous of degree $1$, for all $(u,\bomega) \in \mU \times \mathbb S^{d-1}$ all eigenvalues of $K(u,\bxi)$ are real, semi-simple (i.e. their geometric and algebraic multiplicities coincide) and their multiplicities do not depend on $(u,\bomega)$ (cf. \cite{T91}, Proposition 5.2 C). If this holds for $A^{0}(u)^{-1}A(u,\bxi)$ or $\mB(u,\bxi)$ the respective operator is often called constantly hyperbolic.
\end{remark}

We now fix a homogeneous state $\bar u \in \mU$ and assume the following dissipativity conditions on the coefficient matrices.

{\bf Condition (D).} Matrices $A^j(\bar u), B^{jk}(\bar u)$ have three properties: \par\vskip -1cm
\phantom{x}
\begin{itemize}
	\item[(D1)] For every $\bomega\in \mathbb{S}^{d-1}$,
	all restrictions, as a quadratic form, of                        
	\begin{align*}
		W_1= H(\bar u,\bomega)(A^0(\bar u))^{-1} \big( &- B(\bar u,\bomega) + (A^0(\bar u))^{-1}(A(\bar u,\bomega))
		(A^0(\bar u))^{-1} A(\bar u,\bomega) \\
		&+ C(\bar u,\bomega)(A^0(\bar u))^{-1}A(\bar u,\bomega)  \big) ,
	\end{align*}
	on the eigenspaces $E=J_E^{-1}(\C^n)$ 	
	of
	$$
	W_0=(A^0(\bar u))^{-1}A(\bar u,\bomega)
	$$
	are uniformly negative in the sense that
	$$
	J_E^*\left(
	W_1+W_1^*
	\right)J_E\le -\bar c\
	J_E^*J_E
	\quad\text{with one }\bar c>0.$$
	\item[(D2)]
	For every $\bomega\in S^ {d-1}$,
	all restrictions, as a quadratic form, of
	\begin{equation}
	\mW_1=\mathcal H(\bar u,\bomega) \mathcal A(\bar u,\bomega), \quad \mathcal A(\bar u,\bomega)=
	\begin{pmatrix} 0 & 0 \\ -i A(\bar u,\bomega) & -A^0(\bar u)   \end{pmatrix}
	\end{equation}
	on the eigenspaces $\mE=\mJ_{\mE}^{-1}(\C^{2n})$ of
	\begin{equation}
	\mW_0=  \mathcal B(\bar u,\bomega)
	\end{equation}
	are uniformly negative in the sense that
	$$
	\mJ_{\mE}^*\left(
	\mW_1+\mW_1^*
	\right)\mJ_{\mE}\le -\bar c\ I_\mE
	\quad\text{with one }\bar c>0..
	$$
	\item[(D3)] All solutions
	$(\lambda,\bxi)\in\C\times(\R^d\setminus\{0\})$
	of the dispersion relation of \eqref{shhs1} at $\bar u=0$ have
	$\Rep(\lambda)<0$.
\end{itemize}

\begin{remark}
	Note that as (D) is an open condition there exists a neighbourhood  of $ \bar u$ such that $B^{jk}(u), A^j(u)$ satisfy $(D)$ with $ \bar u$ replaced by $u$ for all $u \in \mU_0$ with $\bar c$ independent of $u$.
\end{remark}

The following remark is useful in the proofs below.

\begin{remark}
	\label{rem:equiv}
	It is straightforward to show that (D1) and (D2) are equivalent to the same conditions with $W_0$, $W_1$ replaced by 
	$$\bar W_0:=H(\bar u,\bomega)^{\frac12}A(\bar u)^{-1}A(0,\bomega)H(\bar u,\bomega)^{-\frac12},~~\bar W_1:=H(\bar u,\bomega)^{-\frac12}W_1H(\bar u,\bomega)^{-\frac12}$$
	 and $\mathcal W_0$, $\mathcal W_1$ replaced by 
	 $$\bar{\mathcal{W}}_0:=\mH(\bar u,\bomega)^{\frac12}\mB(\bar u,\bomega)\mH(\bar u,\bomega)^{-\frac12},~~\bar{\mathcal W}_1:=\mH(\bar u,\bomega)^{\frac12}\mathcal A(\bar u,\bomega)\mH(\bar u,\bomega)^{-\frac12}.$$ 
\end{remark}

From now on we always assume (H$_A$), (H$_B$) and (D). As we could also consider \eqref{shhs1}, \eqref{shhs2} in the variable $u-\bar u$, we can w.l.o.g. restrict our argumentation to the case $\bar u=0$. 

We write \eqref{shhs1} as the first-order in time system
\begin{equation}\label{fos}
	\begin{aligned}
		u_t&=v\\
		v_t&=\sum_{j=1}^d (B^{j0}+B^{0j})(u)v_{x_j}+\sum_{j,k=1}^d B^{jk}(u)u_{x_jx_k}-A^0(u)v-\sum_{j=1}^dA^j(u)u_{x_j}+Q(u,D_{t,x}u)
	\end{aligned}
\end{equation} 
and denote by
\begin{equation}
	\bar{\mM}(u,\bxi)
	:= \begin{pmatrix}
		0 & I_n\\
		M(u,\bxi) & N(u,\bxi)
	\end{pmatrix},
\end{equation}
with
$$
M(u,\bxi)=-iA(u,\bxi)-B(u,\bxi), \quad N(u,\bxi) =  iC(u,\bxi) - A^0(u),
$$
 the Fourier symbol of \eqref{fos}. We also define 
$$\mM(u,\bxi):=\mZ(\bxi)\tilde \mM(u,\bxi) \mZ(\bxi)^{-1}$$
$$\mZ(\bxi)=\begin{pmatrix}
	\langle \bxi \rangle I_{n} & 0\\
	0 & I_n
\end{pmatrix}.$$

First we treat the linearization of \eqref{shhs1} at the reference state $u=0$, i.e.
\begin{equation}
	\label{eq:shhslin}
	\sum_{j=0}^d A_j(0)u_{x_j}=\sum_{j,k=0}^d B^{ij}(0)u_{x_ix_j}.
\end{equation}

Such linear systems were  studied in \cite{FS}, however under the stronger assumptions, that the coefficient matrices are symmetric and $A^0$ is positive definite. Then (H$_A$) is clearly satisfied with $F_A=I_n$ and $H=A^0$. Also, condition  (H$_B$) (b) ibid. requires the existence of a matrix family $\mS:\mathbb S^{d-1} \to \C^{n \times n}$ such that $i\mS(\bomega)\mB(0,\bomega)\mS(\bomega)^{-1}$ is real symmetric. But one can easily check that this can be relaxed to the assumption that $i\mS(\bomega)\mB(0,\bomega)\mS(\bomega)^{-1}$ is hermitian, which is satisfied in the present context for $S(\bomega):=\mH(0,\bomega)^{\frac12}$. Lastly, we want to point out that (D1), (D2) ibid. were stated in the equivalent form mentioned in Remark \ref{rem:equiv}.

We will make plausible below that the weaker conditions in the present work are still sufficient to retrieve the main result of \cite{FS}, namely:

\begin{prop}
	\label{prop:lin}
	There exist a $c>0$ and a family $\bxi\mapsto\mathcal T(\bxi),\R^d\to\C^{2n\times 2n}$
	of linear transformations of $\C^{2n}$ which, together with their inverses $\mathcal T(\bxi)^{-1}$,
	are uniformly bounded, such that 
	\begin{equation} \label{Ntilde2}
	\mT(\bxi)\mM(0,\bxi)\mT^{-1}(\bxi)+(\mT(\bxi)\mM(0,\bxi)\mT^{-1}(\bxi))^*\le -c\rho(\bxi)I_{2n},~~\bxi \in \R^d,
	\end{equation}
	where $\rho(\bxi)=|\bxi|^2/(1+|\bxi|^2)$.
\end{prop}

As outlined in \cite{FS} this brings about the pointwise decay of solutions in Fourier space and thus the following decay estimate for the inhomogeneous linear Cauchy problem.

\begin{coro}
	\label{cor:lin}
For any $s \in \N_0$ there exists $C>0$ such that the following holds: For all $u_0 \in H^{s+1}\cap L^1 $, $u_1 \in H^s\cap L^1$ and  $f \in C([0,T],H^{s} \cap L^1)$ the solution $u$ of
$$f+\sum_{j=0}^d A_j(0)u_{x_j}=\sum_{j,k=0}^d B^{ij}(0)u_{x_ix_j}$$ 
with $u(0)=u_0$, $u_t(0)=u_1$ satisfies
\begin{align*}\|u(t)\|_{s+1}+\|u_t(t)\|_{s} &\le C(1+t)^{-\frac{d}{4}}(\|u_0\|_{s+1}+\|u_0\|_{L^1}+\|u_1\|_s+\|u_1\|_{L^1})\\
	&\quad C\int_0^t (1+t-\tau)^{-\frac{d}{4}}(\|f(\tau)\|_s +\|f(\tau)\|_{L^1})~d\tau
\end{align*}
for all $t \in [0,T]$.
\end{coro}

\begin{proof}[Proof of Proposition \ref{prop:lin}]
As stated above  the proof can be found essentially in \cite{FS}. We just illustrate at which points it has to be slightly modified.

The existence of a bounded family $\mT(\bxi) \subset \operatorname{Gl_{2n}}$\footnote{For $m \in \N$ $\operatorname{Gl}_m$ denotes the space of invertible $m \times m$-matrices.} satisfying \eqref{Ntilde2} is proven separately for the three different regimes $|\bxi| \le r_0$, $r_0 \le |\bxi| \le r_\infty$ and $|\bxi| \ge r_\infty$ for suitable $r_0,r_\infty>0$. In the latter two cases only ($(H)_B$) and conditions (D2), (D3) are used. The symmetry of the matrices plays no role whatsoever.

For small values of $|\bxi|$ writting $\bxi=\xi\bomega$ for $\xi > 0$, $\bomega \in \mathbb{S}^{d-1}$  one finds a bounded family of invertible
 $\mathcal R(\xi,\bomega)$ with $\mathcal R(\xi,\bomega)^{-1}$ also bounded and (supressing the argument $u=0$)
$$
\mathcal R(\xi,\bomega)\bar \mM(\xi\bomega)\mathcal R(\xi,\bomega)^{-1}=\begin{pmatrix}
	 X(\xi,\bomega) & 0\\
	 0 & Y(\xi,\bomega)
\end{pmatrix},$$
where
\begin{align*}
	X(\xi,\bomega)&=i\xi (A^{0})^{-1}A(\bomega)\\
	&+\xi^2(A^0)^{-1} \big(- B(\bomega) + (A^0)^{-1}(A(\bomega))
	(A^0)^{-1} A(\bomega) 
	+ C(\bomega)(A^0)^{-1}A(\bomega)  \big)+\mathcal O(\xi^3)\\
	Y(\xi,\bomega)&=-A^0+\mathcal O(\xi^3).
\end{align*}
This is due to the fact that $A^0(0)$ is invertible and again makes no use of the symmetry. Hence for 
$$\check{\mathcal R}(\xi,\bomega)=\begin{pmatrix} H(\bomega)^{\frac12} & 0\\
	0 & F_A^\frac12 \end{pmatrix}\check{\mathcal R}(\xi,\bomega)$$
we get
$$\check{\mathcal R}(\xi,\bomega)\mM(\xi\bomega)\check{\mathcal R}(\xi,\bomega)^{-1}=\begin{pmatrix}
	i\xi \bar W_0+\xi^2\bar W_1+\mathcal O(\xi^3) & \\
	0 & -F_A^{\frac12}A^0F_A^{-\frac12}.
\end{pmatrix}.$$
with $\bar W_0$, $\bar W_1$ as in Remark \ref{rem:equiv}. Since $F_A^{\frac12}A^0F_A^{-\frac12}$ is positive definite the existence of the family $\mT(\bxi)$ now follows for sufficiently small $\xi$ by condition (D1) and \cite{FS}, Lemma 5.\footnote{Note that in said Lemma it is sufficient to assume that $i\mM(0,\bomega)$ is selfadjoint instead of requiring $i\mM(0,\bomega)$ to be real symmetric.}
\end{proof}

In Section 4 we will see that, given $d \ge 3$, $s>d/2+1$, Corollary \ref{cor:lin} directly implies the decay of a solution to the quasi-linear problem \eqref{shhs1} in  $H^{s-1}$ but only provided that its $H^{s}$-norm  is a-priori known to be small. To close this gap we need to show that the $H^{s}$-norm of a small solution can be bounded by the initial conditions and $L^2$-norms of lower order derivatives. The rest of this section is devoted to a construction preparing such a result. 

In the following for $\bxi \in \R^d$ we write $\bxi=\xi\bomega$ with $\xi=|\bxi| \in [0,\infty), \bomega=\bxi/|\bxi| \in \mathbb{S}^{d-1}$.\\
 For $r>0,u \in \R^n, \bxi \in \R^d$ and $\bomega \in \mathbb{S}^{d-1}$ by $B^n(u,r), B^d(\bxi,r),B^{\mathbb S}(\bomega,r)$ we denote the balls with radius $r$ and center $u,\bxi, \bomega$ with respect to the metrices on $\R^n,\R^d,\mathbb S^{d-1}$. For some $\bomega^* \in \mathbb S^{d-1}$ and $\delta>0$ we use 
$$P(\bomega^*,\delta)=B^{n}(0,\delta) \times [0,\delta) \times B^{\mathbb S}(\bomega^*,\delta)$$. 

\begin{prop}
	\label{lwn}
	There exist $r>0, c_\infty>0$ and  
	  a  mapping $\mD_\infty\in C^\infty(\Omega_\infty,\C^{2n \times 2n})$, $\Omega_\infty:=\bar \mU_0 \times \{\bxi \in \R^d:|\bxi| \ge r^{-1}\}$, $\bar\mU_0:=\overline{B^{n}(0,r)} \subset \mU$, such that:
	\begin{enumerate} 
		\vspace{-0.5\baselineskip}
		\item[(i)] For all $(u,\bxi) \in \Omega_{\infty}$ $$\mD_\infty(u,\bxi)=\mD_\infty(u,\bxi)^* \ge c_\infty I_n,$$
		and  
		$$\mD_\infty(u,\bxi)\mM(u,\bxi)+(\mD_\infty(u,\bxi)\mM(u,\bxi))^* \le -c_\infty I_{2n}.$$
		\item[(ii)] For any $\alpha,\beta \in \N_0^d$ there exist $C_{\alpha\beta}>0$ with
		\begin{equation}
			\label{sym1}|\partial_u^\beta \partial_\xi^\alpha \mD_\infty (u,\bxi)| \le C_{\alpha\beta}\lxi^{-|\alpha|},~~(u,\bxi) \in  \Omega_\infty.
		\end{equation}
	\end{enumerate}
\end{prop}

\begin{proof}
	Consider the mapping $\mK:\mU \times (0,\infty) \times \mathbb S^{d-1} \to \C^{2n \times 2n}$ defined by
\begin{equation} \label{K}
\mK(u,\eta,\bomega)=\begin{pmatrix}  0& I_n \\ 
	 -i\eta A(u,\bomega)-B(u,\bomega) &-i C(u,\bomega)-\eta A^0(u) \end{pmatrix},~~\bomega \in \mathbb S^{d-1}.
\end{equation}
    and $\mH(u,\bomega)$ denote the symmetrizer of $\mB(u,\bomega)$ as in condition (H$_B$) (b). Set
    $$\mW(u,\eta,\bomega):=\mH(u,\bomega)^{\frac12}\mK(u,\eta,\bomega)\mH(u,\bomega)^{-\frac12}.$$
    Since   
    $$\mK(0,0,\bomega) 
    = \begin{pmatrix}  0 & I_n  \\ 
    	 -B(0,\bomega) & i C(0,\bomega) \end{pmatrix} 
    = \mB(0,\bomega)$$   
    and
    $$\frac{\partial \mK}{\partial\eta}(0,0,\bomega)    =  \begin{pmatrix}  0& 0 \\ -i A(0,\bomega) & - A^0(0)   \end{pmatrix}=\mathcal{A}(0,\bomega)$$
     $\mW$ satisfies
    $$
    \mW(0,0,\bomega)=\bar\mW_0,\quad \frac{\partial\mW(0,0,\bomega)}{\partial\eta}
    =\bar \mW_1,$$
    with $\bar \mW_0,\bar\mW_1$ as in Remark \ref{rem:equiv}.
    Now fix $\bomega_0\in \mathbb S^{d-1}$. By virtue of condition (D2) it follows from Lemma 5 in \cite{FS} that there exists $\delta_0>0, c_0>0$ and  $\mT_0 \in C^\infty(P(\bomega^*,\delta_0), \rm{Gl}_{2n})$ with $\mT_0^{-1}$ also bounded such that pointwise on $P(\bomega_0,\delta_0)$
    $$\mT_0\mW \mT_0^{-1}+(\mT_0\mW \mT_0^{-1})^*\le-\tilde c\eta I_{2n}$$
    for some $\tilde c>0$. Hence 
   $\tilde\mD_0:=\mH^{\frac12}\mT_0^*\mT_0\mH^{\frac12}\in C^\infty(P(\delta_0,\bomega_0), \C^{2n \times 2n})$
   satisfies
   $$\tilde \mD_0(u,\xi,\bomega)=\tilde \mD_0(u,\xi, \bomega)^* \ge cI_{2n},~~(u,\xi,\bomega) \in P(\delta_0,\bomega_0)$$
   for some $c>0$ and thus
   $$\tilde \mD_0 \mK+(\tilde \mD_0\mK)^*\le -c \tilde c\eta I.$$
   In conclusion we have shown the following: For each $\bomega \in \mathbb S^{d-1}$ there exist $\delta_{\bomega}>0$, $c_{\bomega}>0$ and   $\mD_{\bomega} \in C^\infty(P(\bomega,\delta_{\bomega}), \C^{2n\times 2n})$ such that for all $(u,\xi,\bar \bomega) \in P(\bomega,\delta_{\bomega})$
   \begin{equation}
   	\label{eq:mDj} 
   	\begin{split}
   		\mD_{\bomega} (u,\eta,\bar \bomega)=\mD_{\bomega}(u,\eta, \bar \bomega)^* &\ge c_{\bomega}I\\
   		\mD_{\bomega} (u,\eta,\bar \bomega)\mK(u,\eta,\bar \bomega)+( \mD_{\bomega}(u,\eta,\bar\bomega)\mK(u,\eta,\bar \bomega))^* &\le -c_{\bomega}\xi^2I.
   	\end{split} 
   \end{equation}
   As $\mathbb S^{d-1}$ is compact we may choose $\bomega_1,\ldots,\bomega_r$ such that
   $$\bigcup_{l=1}^{\bar l} B^{\mathbb S}(\bomega_l,\delta_l/2)=\mathbb S^{d-1}~~(\delta_l:=\delta_{\bomega_l}).$$
   Set $r_0=\min\{\delta_1,\ldots,\delta_r\}$, $c_0=\min\{c_{\bomega_1},\ldots,c_{\bomega_r}\}$. Then for $l=1,\ldots,\bar l$ and $P_l:=B^n(0,r_0) \times [0,r_0) \times B^{\mathbb S}(\bomega_l,\delta_l)$ choose functions $ \phi_l \in C^\infty(\mathbb{S}^{d-1},[0,1])$ with $\supp \phi_l \subset B^{\mathbb S}(\bomega_j,\delta_l)$, $\phi_l=1$ on $B^{\mathbb S}(\bomega_j,\delta_j/2)$ and extend $\mD_l:= \mD_{\bomega_l}$ trivially by $0$ to a function defined on $B^{n}(0,r_0) \times [0,r_0) \times \mathbb S^{d-1}=:\Omega_0$. Define
   $$\mD_0:\Omega_{0} \to \C^{2n \times 2n}: (u,\eta,\bomega)\mapsto \sum_{l=1}^{\bar l} \phi_l(\bomega)\mD_l(u,\eta,\bomega).$$
   Then $\mD_0 \in C^\infty(\Omega_0, \C^{2n \times 2n}$), and $\mD_0(u,\eta,\bomega)$ is hermitian for all $(u,\eta,\bomega) \in \Omega_0$. Furthermore for $(u,\eta,\bomega) \in \Omega_0$ we have $\bomega \in B^{\mathbb S}(\bomega_k,\delta/2)$ for some $k \in \{1,\ldots,\bar l\}$ and thus as $ \mD_l(u,\eta,\bomega) \ge 0$
   $$\mD_0(u,\eta,\bomega)=\sum_{l=1}^{\bar l} \phi_l(\bomega)\mD_l(u,\eta,\bomega) \ge  \mD_k(u,\eta,\bomega) \ge c_0I.$$
   with the same reasoning we see
   $$\mD_0(u,\eta,\bomega) \mK(u,\eta,\bomega)+(\mD_0(u,\eta,\bomega) \mK(u,\eta,\bomega))^* \le -c_0\eta I_{2n},~~(u,\eta,\bomega) \in \Omega_0.$$
   
    Now note that for all $u,\xi,\bomega$
    \begin{align*}
    	\xi \mK(u,1/\xi,\bomega)&=\begin{pmatrix}  0& \xi I_n \\ 
    	 -i A(u,\bomega)-\xi B(u,\bomega) &-i\xi C(u,\bomega)- A^0(U) \end{pmatrix}\\
    	&=\tilde{\mZ}(\xi)\mM(u,\xi\bomega)\tilde{\mZ}(\xi)^{-1},
    \end{align*}
    where 
    $$\tilde{\mZ}(\xi)=\begin{pmatrix} \frac{\lxi}{\xi}I_{n} & 0 \\
    0 & I_n\end{pmatrix}.$$
     As clearly $\tilde \mZ, \tilde \mZ^{-1} \in C^\infty((r_0^{-1},\infty),\C^{2n\times 2n})$ are symmetric and positive definite on $(r_0^{-1}, \infty)$, for  $r:=r_0/2$, $\Omega_\infty:=\overline{B^n(0,r)} \times \{\bxi \in \R^d:|\bxi| \ge r^{-1}\}$ the mapping
    $$\mD_\infty:\Omega_{\infty}\to \C^{2n \times 2n}, (u,\bxi)\mapsto \tilde \mZ(|\bxi|)\mD_0(u,1/|\bxi|,\bxi/|\bxi|)\tilde \mZ(|\bxi|)$$
    is in $C^\infty(\Omega_\infty,\C^{2n \times 2n})$ and for all $(u,\bxi) \in \Omega_{\infty}$	$$\mD_\infty(u,\bxi)=\mD_\infty(u,\bxi)^* \ge c_\infty I_{2n}$$
    for some $c_\infty>0$. Since for $\bxi=\xi \bomega \in \mU_0$
    $$\mD_\infty(u,\bxi)\mM(u,\bxi)=\xi\tilde\mZ(\xi)\mD_0(u,1/\xi,\bomega)\tilde\mZ(\xi)\mK(u,1\xi,\bomega)=\xi\tilde \mZ(\xi)\tilde \mD_0(u,1/\xi,\bomega)\mK(u,1/\xi)\tilde\mZ(\xi),$$
    we also have 
    \begin{align*}
    	\mD_\infty(u,\bxi)\mM(u,\bxi)+(\mD_\infty(u,\bxi)\mM(u,\bxi))^* \le -c_\infty I_{2n}
\end{align*}
    for some $c_\infty>0$.
    
    It remains to verify \eqref{sym1}. First note that the functions $\bxi \mapsto \langle |\bxi| \rangle/|\bxi|$, $\bxi \mapsto \xi_k/|\bxi|$, $k=1,\ldots,d$ and $\bxi \mapsto 1/|\bxi|$ are positive homogeneous of degree $0$ and $-1$, respectively. Thus for any $\alpha \in \N_0^d$ there exists $C_\alpha>0$ such that for  all $\bxi \in \R^d$ with $|\bxi|>2r_0^{-1}$ 
    $$|D^\alpha \mZ(\bxi)|+ |D^{\alpha}(\xi_k/|\bxi|)|+ |D^\alpha(1/|\bxi|)|\le C_\alpha \langle \bxi \rangle^{-|\alpha|}.$$
    Since $\mD_0$ as well as all of its derivatives are bounded on $\overline{B^{n}(0,r_0/2)} \times [0,r_0/2] \times \mathbb S^{d-1}$
    the estimate \eqref{sym1} follows by product and chain rule.
\end{proof}

\newpage
\section{Proof of Theorem \ref{theo:main}}

To begin with, we remark that local well-posednes sof \eqref{shhs1}, \eqref{shhs2} follows from the existing theory for hyperbolic systems of any order \cite{T91}.\footnote{For example, the recent result in \cite{bem212}, which applies to the class we study in Section 5, is of this type.} Our task thus consists in showing that under an a priori smallness assumption the solution satisfies the decay and energy estimates \eqref{eq:decay} and \eqref{eq:energy}, for, w.l.o.g., $\bar u=0$. Then we can extend them globally by standard methods (cf. e.g. \cite{kawa83}, proof of Theorem 3.6). We show the following.

\begin{prop}
	\label{prop:main}
	Consider $d \ge 3$, $s>d/2+1$ and assume (H$_B$), (H$_A$) and (D). Then there exist constants $\mu>0$, $\delta=\delta(\mu)>0$, and $C=C(\mu,\delta)>0$ (all independent of $T$) such that the following holds:
	For all $u_0 \in H^{s+1} $, $u_1 \in H^s$ with $\|u_0\|_{s+1}+\|u_1\|_s <\delta$ and all $u\in C^0([0,T],H^{s+1})\cap C^1([0,T], H^s)$ 
	satisfying \eqref{shhs1}, \eqref{shhs2} and 
	$$\sup_{t \in [0,T]}\|u(t)\|_{s+1}^2+\|u_t(t)\|_{s}^2+\int_0^T \|u(\tau)\|_{s+1}^2+\|u_t(\tau)\|_s^2~d\tau  \le \mu$$ we have for all $t \in [0,T]$
	\begin{align}
		\label{eq:decay1}
		\|u(t)\|_{s}+\|u_t(t)\|_{s-1} &\le C(1+t)^{-\frac{d}{4}}(\|u_0\|_{s}+\|u_0\|_{L^1}+\|u_1\|_{s-1}+\|u_1\|_{L^1}),\\
		\label{eq:decay20}
		\|u(t)\|_{s+1}^2+\|u_t(t)\|_{s}^2&+\int_0^t \|u(\tau)\|_{s+1}^2+\|u_t(\tau)\|_{s}^2 \le C(\|u_0\|_{s+1}^2+\|u_0\|_{L^1}^2+\|u_1\|_{s1}^2+\|u_1\|_{L^1}^2)
	\end{align}
\end{prop}

We split the proof into two parts corresponding to the following two assertions.

\begin{prop} 
	\label{prop:decay}
In the situation of Proposition \ref{prop:main} there exist $\mu>0$, $\delta>0$, and $C>0$ such that the following holds:
For all $u_0 \in H^{s+1}\cap L^1 $, $u_1 \in H^s \cap L^1$ with $\|u_0\|_{s+1}+\|u_1\|_s, \|u_0\|_{L^1}+\|u_1\|_{L^1} <\delta$ and all $u\in C^0([0,T],H^{s+1})\cap C^1([0,T], H^s)$ 
satisfying \eqref{shhs1}, \eqref{shhs2} and 
$$\sup_{t \in [0,T]}\|u(t)\|_{s+1}^2+\|u_t(t)\|_{s+1}^2+\int_0^T \|u(\tau)\|_{s+1}^2+\|u_t(\tau)\|_s^2~d\tau  \le \mu$$ 
\eqref{eq:decay} holds for all $t \in [0,T]$.
\end{prop}

\begin{prop}
	\label{prop:energy}
In the situation of Proposition \ref{prop:main} there exist $\mu>0$,  and $C>0$ such that the following holds:
For all $u_0 \in H^{s+1} $, $u_1 \in H^s$ and all $u\in C^0([0,T],H^{s+1})\cap C^1([0,T], H^s)$ 
satisfying \eqref{shhs1}, \eqref{shhs2} and 
$$\sup_{t \in [0,T]}\|u(t)\|_{s+1}^2+\|u_t(t)\|_{s+1}^2+\int_0^T \|u(\tau)\|_{s+1}^2+\|u_t(\tau)\|_s^2~d\tau  \le \mu$$ we have for all $t \in [0,T]$
\begin{equation}
	\label{eq:wichtig}
	\begin{split}
		& \|u(t)\|_{s+1}^2+\|u_t(t)\|_{{s}}^2+\int_0^t  \|u(\tau)\|_{s}^2+\|u_t(\tau)\|_{{s-1}}^2 d\tau\\
		& \le C(\|u_0\|_{{s+1}}^2+\|u_1\|_{s}^2)+C\int_{0}^t \|u(\tau)\|_{s}^2+\|u_t(\tau)\|_{{s-1}}^2 ~d\tau.
	\end{split}
\end{equation}
\end{prop}

From there Proposition \ref{prop:main} clearly follows by multiplying \eqref{eq:wichtig} with a sufficiently small factor integrating, \eqref{eq:decay} with respect to $t$, and adding the resulting inequalities.

For notational reasons we write the first order representation  \eqref{fos} of \eqref{shhs1} in the compact form
 \begin{equation} 
 	\label{pde}
 U_t=L(u)U+(0,Q(u,D_{x,t}u))^t
 \end{equation}
 with $U=(u,u_t)$,
 $$L(u)=\begin{pmatrix}
 	 0 & I_n\\
 	 \sum_{j,k=1}^d B^{jk}(u)\partial_{x_j}\partial_{x_k}-\sum_{j=1}^d A^j(u)\partial_{x_j} & \sum_{j=1}^d (\bar{B}^{j0}+\bar B^{0j})(u)\partial_{x_j} -A^0.
 \end{pmatrix}$$

\begin{proof}[Proof of Proposition \ref{prop:decay}]
As $s>d/2+1$ we find by Moser type inequalities (cf. \cite{BS} Appendix C and the references therein)
\begin{align*}
	\|(L_2(u)-L_2(0))U\|_{s-1}+\|(L_2(u)-L_2(0))U\|_{L^1}&\le C_\mu\|u\|_{s-1}(\|u\|_{s+1}+\|u_t\|_{s}), 
\end{align*}
where $L(u)U=(U_2,L_2(u)U)$.  Furthermore 
$$\|Q(u,D_{x,t}u))\|_{s-1}+\|Q(u,D_{x,t}u)\|_{L^1} \le C_\mu\|u\|_s \|u_t\|_{s-1}.$$
Now  writing system \eqref{pde} as $L(0)U=(0,L_2(0)-L_2(u)+Q(u,D_{x,t}u))$ and applying Corollary \ref{cor:lin} to $f=(L_2(0)-L_2(u))+Q(u,D_{x,t}u)$ with $s$ replaced by $s-1$ yields
\begin{equation}
	\label{eq:decay2}
	 \begin{aligned}
	\|u(t)\|_{s}&+\|u_t(t)\|_{s-1} \le C(1+t)^{-\frac{d}{4}}(\|u_0\|_{s}+\|u_0\|_{L^1}+\|u_1\|_{s-1}+\|u_1\|_{L^1})\\
&+C_\mu \sup_{\tau \in [0,t]} (\|u(\tau)\|_{s+1}+\|u_t(\tau)\|_s)\int_0^t(1+t-\tau)^{-\frac{d}{4}}(\|u(\tau)\|_s+\|u_t\|_{s-1} )d\tau.
\end{aligned}\end{equation}
As $t\to (1+t)^{-\frac{d}{4}}$ is square-integrable over $[0,\infty)$ for $d \ge 3$ this gives \eqref{eq:decay} as in e.g. \cite{kawa83}, proof of Proposition 3.3.
\end{proof}

\begin{proof}[Proof of Proposition \ref{prop:energy}]
From now $C_\mu$ always denotes some constant depending monotonically increasing on $\mu$, whose concrete value may change at every instance.
	
For $0<\eps<1$ let $J_\eps$ be the Friedrichs mollifier and set $V=(\Lambda u,u_t)$, $W:=W_{\eps}:=\Lambda^sJ_\eps (\Lambda  u,u_t)$ and
\begin{align*}
	\mM_u(x,\xi)&=\mM_{u(t)}(x,\xi)=\mathcal{M}(u(t,x),\xi)\\
	&=\begin{pmatrix}
		0 & \lxi I_n\\
		\big(-B(u,\xi)-A(u,\xi)\big)\lxi^{-1} &  C(u,\xi)-A^0(u)\end{pmatrix}.\\
\end{align*}
We start with the following observation.
\begin{lemma}
 $W$ satisfies the differential equation
	 \begin{equation} 
	 	\label{pde3}
	 	W_t=\Op[\mM_u] W+R_1,
	 \end{equation}
 	for some $R_1 \in L^2$ satisfying
 	\begin{equation}
 		\label{estR2}
 		\|R_1\| \le C_\mu \|V\|_{s}^2+C\|V\|_{s-1}
 	\end{equation}
\end{lemma}

\begin{proof}
	Set  
	$$\tilde L(u):=\begin{pmatrix}
		\Lambda I_n&0\\ 0& I_n
	\end{pmatrix}L(u)\begin{pmatrix}
		\Lambda^{-1}I_n &0 \\ 0& I_n
	\end{pmatrix}.$$ 
	Then
	\begin{equation}
		\label{pde2}
		V_t=\Op[\mM_u]V+\tilde R_1
	\end{equation}
	where
	$$\tilde R_1=(\tilde L(u)-\Op[\mM_u])V+(0,Q(u,D_{x,t}u)).$$
	As we have already seen in the proof of Proposition \ref{prop:decay} (now with $s-1$ replaced by $s$)
	$$\|Q(u,D_{x,t}u)\|_{s} \le C_\mu\|V\|_{s}^2.$$
	By Lemma \ref{lem:is2}
	$$\|(\tilde L(0)-\Op[\mM_0])V\|_s \le C\|V\|_{s-1}$$
	and due to Lemma \ref{lem:is2} (iii) all terms appearing in 
	$$(\tilde L(u)- \tilde L(0)-\Op[\mM_u-\mM_0])V$$
	are of the form $(a(u)-\Op[a_u])f$, where $a$ is a smooth function with $a(0)=0$ and $f \in \{\partial^l_t\partial^\beta_xu|~l\le 1, l+|\beta| \le 2\} \subset H^{s-1} \hookrightarrow L^\infty$. Hence Lemma \ref{lem:diffmult} yields
	$$\|(\tilde L(u)- \tilde L(0)-\Op[\mM_u-\mM_0])V\|_s \le C_\mu (\|u\|_s \|V\|_{s}+\|V\|_{s-1}).$$
	In conclusion we have shown
	\begin{equation}
		\label{estR1}
		\|\tilde R_1\|_s \le C_\mu (\|V\|_{s}^2+\|V\|_{s-1}).
	\end{equation}
	
	Now apply $\Lambda^s J_\eps$ to \eqref{pde2} and obtain
	\begin{equation} 
		W_t=\Op[\mM_u] W+R_1,
	\end{equation}
	where 
	\begin{align*}
		R_1&=[\Lambda^sJ_\eps, \Op[\mM_u]]V+\Lambda^sJ_\eps \tilde R_1
	\end{align*}
	Note that $(J_\eps)_{\eps \in (0,1)}$ is a family of pseudo-differential operators, constant with respect to $x$, with symbols uniformly bounded in $S^0$. Thus we get from \eqref{estR1}
	$$\|\Lambda^sJ_\eps R_1\| \le C_\mu\|V\|_{s}^2+ C\|V\|_{s-1}$$
	and from Proposition \ref{prop:central} (iii)
	$$\|[\Lambda^sJ_\eps, \Op[\mM_u]]V\| \le C_\mu\|u\|_s\|V\|_s+C\|V\|_{s-1},$$
	which proves the assertion.
\end{proof}

Next, let $\mD_\infty \in C^\infty(\overline{B^n_r(0)} \times \{\xi \in \R^d: |\xi| \ge r\}, \C^{2n \times 2n})$ be the mapping constructed in Proposition \ref{lwn} and extend it trivially by zero to a function defined on $\overline{B^n_r(0)} \times \R^d:=\mU_0 \times \R^d$. Choose a function $\phi \in C^\infty(\R^d)$, with $0 \le \phi \le 1$, $\phi(\xi)=0$ for $|\xi| \le 2r$ and $\phi(\xi)=1$ for $|\xi| \ge 3r$. Set $$\mD(v,\xi):=\phi(\xi)\mD_\infty(v,\xi),\quad (v,\xi) \in \mU_0 \times \R^d$$
Let $\mu$ be sufficiently small such that $u(t,x) \in \overline{B^n_r(0)}$ for all $(t,x) \in [0,T] \times \R^d$ and define
$$\mD_u(x,\xi):=\mD_{u(t)}(x,\xi)=\mD(u(t,x),\xi), \quad (t,x,\xi) \in [0,T] \times \R^d \times \R^d.$$
Choose another function $\psi \in C^\infty(\R^d)$, with $0 \le \psi \le 1$,  $\psi(\xi)=0$ for $|\xi| \ge 5r$, $\psi(\xi)=1$ for $|\xi| \le 4r$ and define
$$\tilde \mD_u(x,\xi)=\mD_u(x,\xi)+\psi(\xi)I_{2n}.$$
\begin{lemma}
	\label{lem:proof2}
	The family of operators $(\mG_{u(t)})_{t \in [0,T]}$ defined by
	$$\mathcal\mG_{u(t)}:=\frac12(\Op[\tilde\mD_{u(t)}]+\Op[\tilde\mD_{u(t)}]^*)+\op[\tilde D_0]-\Op[\tilde D_0]$$
	is self-adjoint and uniformly positive definite in $\mL(L^2)$ for $\mu$ sufficiently small.  Furthermore 
		$$\frac12\frac{d}{dt}\big\langle \mG_u W,W\rangle = \Rep\langle \mG_u \Op[\mM_u]W,W\rangle+R_2,$$
	for some $R_2 \in \R$ with
	$$|R_2| \le C_\mu\|W\|(\|V\|_s^2+\|V\|_s\|W\|+\|V\|_{s-1}).$$
\end{lemma} 

\begin{proof}
 By Proposition \ref{lwn} $\tilde \mD, \mD \in S^0(\mU)$ and $\tilde\mD_u=\tilde\mD_u^*$ is uniformly positive definite. In particular, $\op[\tilde\mD_0]=\op[\tilde \mD_0]^*$ is a self-adjoint and uniformly positive definite operator on $\mL(L^2)$ (cf. Lemma \ref{lem:is2}). Due to ibid. also $\Op[\tilde\mD_0]^*=\Op[\tilde \mD_0]$, i.e.
$$\mG_u=\op[\tilde\mD_0]+\frac{1}{2}(\Op[\tilde\mD_u-\tilde\mD_0]+\Op[\tilde\mD_u-\tilde\mD_0]^*).$$
 Proposition \ref{prop:central} (i) gives
$$\|\Op[\tilde\mD_u-\tilde\mD_0]\|_{\mathcal L(L^2)} \le C_\mu\|u\|_s,$$
which yields the first assertion.

Now apply $\mG_u$ to \eqref{pde3}, take the $L^2$ scalar product with $W$ and consider the real part to find
\begin{equation}
	\label{est100}
	\begin{split}
	\Rep\langle \mG_u W_t,W\rangle&=\Rep\langle \mG_u \Op[\mM_u]W,W\rangle+\Rep\langle \mG_uR_1,W\rangle:=\Rep\langle \mG_u \Op[\mM_u]W,W\rangle+R_{21}.
	\end{split}
\end{equation}
Due to \eqref{estR2} and $\|\mathcal G_u\|_{\mL(L^2)} \le C_\mu$, 
\begin{equation}
	\label{estR}
	\|R_{21}\| \le C_\mu\|W\|(\|V\|_{s}^2+\|V\|_{s-1}).
\end{equation}
As $\mG_u$ is self-adjoint we get
\begin{equation}
	\label{estT}
	\Rep \langle \mG_u W_t,W\rangle=\frac12\frac{d}{dt}\big\langle \mG_u W,W\rangle-\Rep\big\langle \big(\frac{d}{dt}\mG_u\big) W,W\rangle
\end{equation}
and \ref{prop:dev} (iv) yields
\begin{equation} 
	\label{estT2}
	2\|\frac{d}{dt}\mG_u\|_{\mathcal L(L^2)} \le \big\|\frac{d}{dt}\Op[\tilde\mD_u]\big\|_{\mathcal L(L^2)}\le C_\mu\|u_t\|_s.
\end{equation}
The second statement then clearly follows from \eqref{est100}-\eqref{estT2}.
\end{proof}

The last step consists in showing the following.

\begin{lemma}
	\label{lem:proof3}
	It holds
	$$\Rep \langle \mG_u\Op[\mM_u]W,W \rangle \le -c\|W\|^2+C_\mu\|W\|^2(\|u\|_{s+1}^\frac12+\|u\|_s)+C_\mu\|W\|_{-1}^2).$$
\end{lemma}

From Lemmas \ref{lem:proof2}, \ref{lem:proof3} we obtain
\begin{equation}
	\label{eq:last}
	\frac12\frac{d}{dt}\langle \mG_uW,W\rangle +c\|W\|^2 \le C_\mu\|W\|(\|V\|_s^2+\|V\|_s\|W\|+\|V\|^{\frac12})+C_\mu(\|V\|_{s-1}^2+\|W\|_{-1}^2).
\end{equation}

As $\Lambda^{-k}W=\Lambda^{-k}W_\eps \to V$ as $\eps \to 0$ uniformly with respect to $t$ for $0 \le k \le s$ and $\mG_u$ is uniformly postive definite, we find by integrating \eqref{eq:last}
$$\|V(t)\|_s^2+\int_0^t \|V\|_s^2 ~d\tau \le C_\mu(\|V(0)\|+\int_0^t \|V(\tau)\|_s^3+\|V(\tau)\|_s^{\frac52}+\|V(\tau)\|_{s-1})d\tau,~~t \in [0,T],$$
which yields the assertion since $\|V\|_s^2=\|u\|_{s+1}^2+\|u_t\|^2$.
\end{proof}

\begin{proof}[Proof of Lemma \ref{lem:proof3}]
	
Set $\kappa:=\op[\tilde \mD_0]-\Op[\tilde \mD_0]$, which is infinitely smoothing. Then
$$\mG_u=\frac12(\Op[\tilde \mD_u]+\Op[\tilde \mD_u]^*)+\kappa.$$ 
As $\|\mM_u\|_{\mathcal L(H^{l},H^{l-1})} \le C_\mu$, $l \in \R$, due to Proposition \ref{prop:central} (i)  we find
\begin{align*}
		\Rep\langle \kappa\Op[\mM_u]W,W \rangle &\le C_\mu \|W\|_{-1}^2
\end{align*}
By construction $\tilde \mD_u=\tilde \mD_u^*$ and thus \ref{prop:central} (ii) yields
\begin{align*}
	\Rep\big\langle(\frac12(\Op[\tilde\mD_u]^*-\Op[\mD_u])\Op[\mM_u]W,W\big\rangle\le C_\mu\|u\|_s  \|W\|^2.
	\end{align*}
Next note that $\tilde \mD_u(x,\cdot)-\mD_u(x,\cdot)$ is compactly supported with support not depending on $t$. Therefore
 $\Op[\tilde \mD_u-\mD_u]$ is infinitely smoothing and 
$$\Rep \langle \Op[\tilde \mD_u-\mD_u]\Op[\mM_u] W,W \rangle \le C_\mu\|W\|_{-1}^2.$$
In conclusion
\begin{equation} 
	\label{est1}
	\begin{split}
	&\Rep\langle \mG_u\Op[\mM_u] W,W\rangle=\Rep \langle \Op[\mD_u]\Op[\mM_u]W,W\rangle+\Rep\langle \kappa\Op[\mM_u]W,W \rangle\\
	&+\frac12\Rep\langle(\Op[\tilde\mD_u]^*-\Op[\tilde\mD_u])\Op[\mM_u])W,W\rangle+\Rep\langle \Op[\tilde \mD_u-\mD_u]\mM_u W,W \rangle\\
	& \le \Rep \langle \Op[\mD_u]\Op[\mM_u]W,W\rangle+C_\mu(\|u\|_s\|W\|^2+\|W\|_{-1}^2)
	\end{split}
\end{equation}
By Proposition \ref{prop:central} (iii)
$$\|(\Op[\mD_u]\Op[\mM_u]-\Op[\mD_u\mM_u])W\| \le C_\mu \|u\|_s \|W\|+C\|W\|_{-1}.$$
Hence
\begin{equation}
	\label{est2}
	\begin{split}
	\Rep \langle \Op[\mD_u]\Op[\mM_u]W,W\rangle &\le \Rep \langle \Op[\mD_u\mM_u]W,W \rangle+C_\mu \|u\|_s\|W\|^2+C\|W\|_{-1}^2.
	\end{split}
\end{equation}
Set $\mathcal X_u:=\mD_u\mM_u+c_\infty/2I_{2n}$ with $c_\infty$ as in Lemma \ref{lwn}. Note that $c_\infty$  does not depend on $\mu$. 
Since $\Op[I_{2n}]-{\rm Id_{\mL^2}}$ is infinitely smoothing we conclude
\begin{equation} 
	\label{rhoxi1}
	\Rep \langle \Op[\mD_u\mM_u]W,W \rangle\le\Rep \langle \Op[\mathcal X_u]W,W\rangle-\frac{c_\infty}{2}\|W\|^2+C\|W\|_{-1}^2.
\end{equation}
By Proposition \ref{lwn}
$$\mathcal X_u(x,\xi)+\mathcal X_u^*(x,\xi) =\mD_u\mM_u(x,\xi)+(\mD_u\mM_u)(x,\xi)^*+c_\infty \le 0,$$
for $x \in \R^d$ und $\xi \in \R^d$ with $|\xi| \ge 3r$.
Since $u \in H^{s+1}$ and $s+1 \ge d/2+2$, Proposition \ref{prop:gading} applied to $-\mathcal X_u$ gives 
\begin{equation}
	\label{eq:garding100}
	\Rep \langle \Op[\mathcal X_u]W,W \rangle \le C_\mu(\|u\|_{s+1}^\frac12\|W\|^2+\|W\|_{-1}).
\end{equation}
\eqref{rhoxi1} and \eqref{eq:garding100} lead to
\begin{equation}
	\label{est4}
	\Rep \langle \Op[\mD_u\mM_u]W,W \rangle\le-c\|W\|^2+C_\mu(\|u\|_{s+1}^{\frac12}\|W\|^2+\|W\|_{-1}^2)
\end{equation}
for $c$ independent of $u$. Clearly the assertion follows from \eqref{est1}, \eqref{est2}, \eqref{rhoxi1} and \eqref{est4}.
\end{proof}

\section{A class of examples from dissipative relativistic fluid dynamics}

We consider the Euler-augmented Navier-Stokes formulation of dissipative relativistic fluid dynamics on flat Minkowski space-time derived in \cite{Fre20} as a generalization of a model proposed in \cite{bem18}. For barotropic fluids it consists of a system of four equations which, using Einstein's summation convention, read 
\begin{equation} 
	\label{eq:baro}
	A^{\alpha\beta\gamma}(\psi^\eps)\frac{\partial \psi_{\gamma}}{\partial x^\delta}=\frac{\partial}{\partial x^\beta}\big(B^{\alpha\beta\gamma\delta}(\psi^\eps)\frac{\partial \psi_\gamma}{\partial x^\delta}\big),\quad \alpha=0,1,2,3,
\end{equation}
where all Greek indices run from $0$ to $3$, $A^{\alpha\beta\gamma}, B^{\alpha\beta\gamma\delta}$ are contravariant tensors and the unknown function $\psi^\epsilon=(\psi^0,\psi^1,\psi^2,\psi^3)^t$ determining the state of the fluid  is a $4$-vector with respect to the Minkowski-metric of flat space-time. More specifically  $\psi^\epsilon=u^\epsilon/\theta$ with $u^\epsilon$ being the $4$-velocity, $\theta$  the temperature of the fluid. We show that the results of the present work imply non-linear stability of the homogeneous reference state $\bar{\psi}=\bar u^\epsilon/\bar \theta$, where $\bar u^\epsilon=(1,0,0,0)$ represents the fluid's rest frame and $\bar \theta>0$ is a constant temperature. 

For a fluid with equation of state $p=\rho/r, 1\le r<\infty$, $p$ being the pressure, $\rho$ the specific internal energy, the coefficent matrices evaluated at $\bar \psi$ are given by \cite{Fre20} (w.lo.g. assume $\bar \theta=1$)\footnote{Here $e^1,e^2,e^3$ and $\delta^{ij}$ denote the conanical basis of $\R^3$  and the Kronecker symbol, respectively.}
\begin{align*}
A^0(\bar \psi)&=\begin{pmatrix}
	r & 0\\
	0& I_3
\end{pmatrix},\quad A^j(\bar \psi)=\begin{pmatrix}
 0 & (e^j)^t\\
 e^j & 0
\end{pmatrix},\quad
B^{00}(\bar \psi)=\begin{pmatrix}
	-r^2\mu &0 \\
	0 & -\nu I_3
\end{pmatrix},\\
B^{0j}(\bar \psi)&=B^{j0}(\bar \psi)=\frac12\begin{pmatrix} 0 &-(\mu r+\nu)(e^j)^t \\
-(\mu r+\nu)e^j &0\\
\end{pmatrix},\\
B^{ij}(\bar \psi)&=\begin{pmatrix}
	-\nu\delta^{ij} & 0\\
	0 & \eta \delta^{ij}+\frac12(-\mu+\frac13\eta+\zeta)(e^i\otimes e^j+e^j\otimes e^i)
\end{pmatrix}, \quad i,j=1,2,3,
\end{align*}
where $\eta,\zeta >0$ quantify the fluid's viscosity, $\nu,\mu >0$ with $\mu>\tilde{\eta}:=\frac43\eta+\zeta$  reflect a frame change and $A^\beta(\psi^\eps):=(A^{\alpha\beta\gamma}(\psi^\eps))_{0\le \alpha,\gamma \le 3}, B^{\beta\gamma}(\psi^\eps):=(B^{\alpha\beta\gamma\delta}(\psi^\eps))_{0 \le \alpha,\gamma \le 3}$, $\beta,\delta =0,\ldots,3$.

We do not give the detailed non-linear formulation at this point and just refer to \cite{Fre20}. The only information we need for the argumentation below is the fact that for all $\beta,\delta=0,\ldots,3$ and all states $\psi^\eps$ the coefficient matrices $A^\beta(\psi^\eps)$, $B^{\beta\delta}(\psi^\eps)$, $\beta,\delta=0,\ldots,3$ are  symmetric (cf. ibid.).

We show (H$_A$), (H$_B$), (D) for the matrices $(-B^{00})^{-1}B^{\beta\delta}$, $(-B^{00})^{-1}A^\beta$.

(H$_A$) is straightforward: As $-B^{00}(\psi^\epsilon)$, $A^0(\psi^\epsilon)$ are positive definite at $\psi^\eps=\bar \psi$ and symmetric for all states they are symmetric positive definite also in a neighbourhood of $\bar \psi$. Thus (H$_A$) (a) is satisfied with $F_A(u)=-B^{00}(u)$ and (H$_A$) (b) with $H(u)=A^0(u)$. 

Regarding (H$_B$) Freistühler proved ibid. that at the reference state $\psi^\eps=\bar \psi$ for each $\bomega \in \mathbb S^{2}$ the matrix 
$$\tilde{\mathcal B}(\psi^\eps,\bomega)=\begin{pmatrix}
	0 & I_4\\
	-(-B^{00})^{-\frac12}B(\psi^\eps,\bomega)(-B^{00})^{-\frac12} & i(-B^{00})^{-\frac12}C(\psi^\eps,\bomega)(-B^{00})^{-\frac12}
\end{pmatrix},$$
where
\begin{align*}
	B(\psi^\eps,\bomega)&=\sum_{ij=0}^dB^{ij}(\psi^\eps)\omega_i\omega_j,\quad C(\psi^\eps, \bomega)=2\sum_{j=0}^dB^{0j}(\psi^\eps)\omega_j, ~~\bomega=(\omega_1,\ldots,\omega_2) \in \mathbb{S}^2,\\
\end{align*}
 has four simple and two semi-simple  purely imaginary eigenvalues. This is then also true for
 $$
 	\mB(\psi^\eps,\bomega):=\begin{pmatrix}
 		0 & I_4\\
 		(-B^{00})^{-1}B(\psi^\eps,\bomega) & i(-B^{00})^{-1}C(\psi^\eps,\bomega),
 	\end{pmatrix}=\mT^{-1} \tilde{\mB}(\psi^\eps,\bomega)\mT
$$
with $\mT=\operatorname{diag}((-B^{00})^{\frac12},(-B^{00})^{\frac12})$.
Now in the present context the geometric multiplicities of purely imaginary eigenvalues of $\mB(\psi^\eps,\bomega)$  are state invariant properties. Therefore there exists a symbolic symmetrizer of $\mB$  due to Remark \ref{rem:multis}. \\
To see this invariance note that (even in the general setting in Section 3) the eigenvectors $v=v(u,\bomega) \in \C^{2n}\setminus\{0\}$ to an eigenvalue $\lambda=\lambda(u,\bomega) \in \C$ of $\mathcal{B}(u,\bomega)$ are exactly of the form
$v=(v_1, \lambda v_1)$ with $v_1 \in \C^n$ such that $e^{\lambda t+i\bxi}v_1$ is a plane wave solution to the linearization of \eqref{shhs1} at $u$.  As \eqref{eq:baro} is a covariant expression, $e^{\lambda t+i\xi}v$ being a plane wave solution with $\lambda \in i\R$  is also a covariant property (cf. e.g. \cite{FRT}).

It remains to show (D1), (D2), (D3). In the following we only consider matrices evaluated at $\bar \psi$.  The Fourier-symbols correpsonding to the differential operators in \eqref{eq:baro} are given by 
\begin{align*} A(\bomega)&=\sum_{j=1}^dA^{j}\omega_j=\begin{pmatrix}
		0 & \bomega^t\\
		\bomega & 0
	\end{pmatrix},~~B(\bomega)=\sum_{j,k=1}^dB^{jk}\omega_j\omega_k=\begin{pmatrix}
			-r^2\mu & 0\\
			0& \eta +(-\mu +\frac13 \eta+\zeta)\bomega \otimes \bomega
		\end{pmatrix},\\
		C(\bomega)&=2\sum_{j=1}^dB^{0j}\xi_j\begin{pmatrix}
			0 & -(\mu r+\nu)\bomega^t\\
			-(\mu r+\nu)\bomega &0
		\end{pmatrix},~~\bomega=(\omega_1,\omega_2,\omega_3) \in \mathbb S^{d-1}.
\end{align*}
It is straightforward to see that for any $\bomega \in \mathbb S^{d-1}$ the matrices $A^0,A^j(\bomega), B^{00}, B^{jk}(\bomega), C(\bomega
)$ all decompose in sense of linear operators as $A^0=A^0_l\oplus A^0_t$, $A(\bomega)=A_l\oplus A_t$, $B^{00}=B^{00}_l\oplus B^{00}_t$,  $B(\bomega)=B_l \otimes B_t$, $C(\bomega)=C_l \oplus C_t$ with respect to the orthogonal decomposition $\C^4=(\C \times \bomega \C)\oplus(\{0\} \times \{\bomega\}^{\perp})$. Thus we can verify the conditions for $A^0_l, A_l, B^{00}_l, B_l, C_l$ and $A^0_t, A_t, B^{00}_t, B_t, C_t$ separately. We have
$$ A^0_t=I_2,~~A_t=0,~~B^{00}=-\nu I_2,~~B_t=\eta I_2,~~C_t=0.$$
As $\eta>0$, these matrices correspond to coefficients of  damped wave equations and it is well-known that such equations satisfy (D). One can also check this easily by virtue of \cite{FS}, Theorem 4 and Lemma 5.

Next
\begin{align*}A^0_l&=\begin{pmatrix} 
	r & 0\\
	0& 1 
\end{pmatrix},\quad A_l=\begin{pmatrix}
0 & 1\\
1&0
\end{pmatrix},\\
B^{00}_l&=\begin{pmatrix}
	-r^2\mu & 0\\
	0 & -\nu
\end{pmatrix},\quad B_l=\begin{pmatrix}
-\nu & 0\\
0 & \tilde \eta -\mu
\end{pmatrix},\quad C_t=\begin{pmatrix}
0 & -(\mu r+\nu)\\
-(\mu r+\nu) &0
\end{pmatrix}.
\end{align*}
It was shown in \cite{FS} that 
$$\tilde A^j=(-B^{00}_l)^{-\frac12}A^j_l(-B^{00}_l)^{-\frac12},\quad \tilde B^{jk}_l=(-B^{00}_l)^{-\frac12}B^{jk}_l(-B^{00}_l)^{-\frac12}$$
 satisfy (D). But then also $\check A^j:=(-B^{00}_l)^{-1}A^j_l, \check B^{jk}:=(-B^{00}_l)^{-1}B^{jk}_l$ satisfy (D). 

To see this note that $\check A^j=S\tilde A^jS^{-1}$, $\check B^{jk}=S\tilde B^{jk}S^{-1}$ with $S=(-B^{00})^{\frac12}$. Hence we have $\check W_0=S \tilde W_0 S^{-1}$ for $\check W_0$ and $\tilde W_0$ as in (D1) for the matrices $\check A^j, \check B^{jk}$ and $\tilde A^j, \tilde B^{jk}$, respectively. Further  the  symbolic symmetrizer $\tilde H=\tilde A^0$ of $(\tilde A^0)^{-1}\tilde A(\bomega)$ the matrix $\check H=S^{-1}\tilde A^0 S^{-1}$ is a symbolic symmetrizer for $(\check A^0)^{-1}\check A(\bomega)$. This yields $\check W_1=S^{-1}\tilde W_1S^{-1}$ with $\check W_1, \tilde W_1$ as in (D1) for the respective matrices. If now $v$ is an eigenvector of $\check W_0$, $S^{-1}v$ is an eigenvector of $\tilde W_0$ and as $\tilde A^j,\tilde B^{jk}$ satisfy (D1) we get
$$\langle (\tilde W_1+\tilde W_1^*)S^{-1}v,S^{-1}v\rangle \le -c|S^{-1}v|^2 \le -\check c|v|^2,$$
i.e. $\langle (\check W_1+\check W_1^*)v,v \rangle \le -\check c|v|^2,$ which proves (D1) for $\check A^j, \check B^{jk}$.

(D2) follows analogously since with $\mS=\operatorname{diag}((-B^{00})^\frac12,(-B^{00})^{\frac12})$ the matrix $\mS^{-1} \tilde\mH(\bomega) \mS^{-1}$ is a symbolic symmetrizer for $\check \mB(\bomega)$ if $\tilde \mH(\bomega)$ is a symbolic symmetrizer for $\tilde \mB(\bomega)$.

Lastly, (D3) is satisfied trivially, as the matrices introduce equivalent systems of PDEs and thus solutions to the dispersion relation are identical for the two systems. 

\section*{Statements and declarations} 

\textbf{Funding.} This work was supported by DFG Grants No.\ FR 822/10-1, 10-1/2)

\textbf{Competing interests.} The author has no competing interests to declare that are relevant to the content of this article.

\textbf{Acknowledgement.} The author would like to sincerely thank Heinrich Freistühler for his highly helpful suggestions and comments as well as many fruitful discussions.

\bibliographystyle{abbrv}
\bibliography{Manuscript_Sroczinski}

\begin{thebibliography}{10}

\bibitem{bem18}
F.~S. Bemfica, M.~M. Disconzi, and J.~Noronha.
\newblock Causality and existence of solutions of relativistic viscous fluid
  dynamics with gravity.
\newblock {\em Phys. Rev. D}, 98(10):104064, 2018.

\bibitem{bem21}
F.~S. Bemfica, M.~M. Disconzi, and J.~Noronha.
\newblock First-order general relativistic viscous fluid dynamics.
\newblock {\em Phys. Rev. X}, 12(2):021044, 2022.

\bibitem{bem212}
F.~S. Bemfica, M.~M. Disconzi, C.~Rodriguez, and Y.~Shao.
\newblock Local existence and uniqueness in {S}obolev spaces for first-order
  conformal causal relativistic viscous hydrodynamics.
\newblock {\em Commun. Pure Appl. Anal.}, 20(6):2279--2290, 2021.

\bibitem{BS}
S.~Benzoni-Gavage and D.~Serre.
\newblock {\em Multidimensional {H}yperbolic {P}artial {D}ifferential
  {E}quations: {F}irst-order {S}ystems and {A}pplications}.
\newblock Oxford {M}athematical {Monographs}. Oxford University Press, Oxford,
  2006.

\bibitem{bia07}
S.~Bianchini, B.~Hanouzet, and R.~Natalini.
\newblock Asymptotic behavior of smooth solutions for partially dissipative
  hyperbolic systems with a convex entropy.
\newblock {\em Comm. Pure Appl. Math.}, 60(11):1559--1622, 2007.

\bibitem{bon81}
J.~M. Bony.
\newblock Calcul symbolique et propagation des singularit\'{e}s pour les
  \'{e}quations aux d\'{e}riv\'{e}es partielles non lin\'{e}aires.
\newblock {\em Ann. Sc. Ec. Norm. Sup. {(4)}}, 14(2):209--246, 1981.

\bibitem{bou83}
G.~Bourdaud.
\newblock {\em Sur les op\'{e}rateurs pseudo-diff\'{e}rentiels \`{a}
  coefficients peu r\'{e}guliers}.
\newblock PhD thesis, Univ. de Paris-Sud, 1983.

\bibitem{bou88}
G.~Bourdaud.
\newblock Une alg\`{e}bre maximale d'op\'{e}rateurs pseudo-diff\'{e}rentiels.
\newblock {\em Comm. Partial Diff. Eq.}, 13(9):1059--1083, 1988.

\bibitem{che94}
G.-Q. Chen, C.~D. Levermore, and T.-P. Liu.
\newblock Hyperbolic conservation laws with stiff relaxation terms and entropy.
\newblock {\em Comm. Pure Appl. Math.}, 47(6):787--830, 1994.

\bibitem{che14}
J.~Chen and W.~Wang.
\newblock The point-wise estimates for the solution of damped wave equation
  with nonlinear convection in multi-dimensional space.
\newblock {\em Commun. Pure Appl. Anal.}, 13(1):307--330, 2014.

\bibitem{dha11}
P.~M.~N. Dharmawardane, T.~Nakamura, and S.~Kawashima.
\newblock Global solutions to quasi-linear hyperbolic systems of
  viscoelasticity.
\newblock {\em Kyoto J. Math.}, 51(2):467--483, 2010.

\bibitem{Fre20}
H.~Freist\"{u}hler.
\newblock A class of {H}adamard well-posed five-field theories of dissipative
  relativistic fluid dynamics.
\newblock {\em J. Math. Phys.}, 61(3):033101, 2020.

\bibitem{FRT}
H.~Freist\"{u}hler, M.~Reintjes, and B.~Temple.
\newblock Decay and subluminality of modes of all wave numbers in the
  relativistic dynamics of viscous and heat conductive fluids.
\newblock {\em J. Math. Phys.}, 62(5):053101, 2021.

\bibitem{FS}
H.~Freist\"{u}hler and M.~Sroczinski.
\newblock A class of uniformly dissipative symmetric hyperbolic-hyperbolic
  systems.
\newblock {\em J. Differ. Equ.}, 288:40--61, 2021.

\bibitem{ft14}
H.~Freist\"uhler and B.~Temple.
\newblock Causal dissipation and shock profiles in the relativistic fluid
  dynamics of pure radiation.
\newblock {\em Proc. R. Soc. A}, 470(2166):20140055, 2014.

\bibitem{ft16}
H.~Freist\"uhler and B.~Temple.
\newblock Causal dissipation for the relativistic dynamics of ideal gases.
\newblock {\em Proc. R. Soc. A}, 473(2201):20160729, 2017.

\bibitem{ft18}
H.~Freist\"uhler and B.~Temple.
\newblock Causal dissipation in the relativistic dynamics of barotropic fluids.
\newblock {\em J. Math. Phys.}, 59(6):063101, 2018.

\bibitem{GR}
P.~G\'{e}rard and J.~Rauch.
\newblock Propagation de la r\'{e}gularit\'{e} locale de solutions
  d'\'{e}quations hyperboliques non lin\'{e}aires.
\newblock {\em Ann. Inst. Fourier (Grenoble)}, 37(3):65--84, 1987.

\bibitem{han03}
B.~Hanouzet and R.~Natalini.
\newblock Global existence of smooth solutions for partially dissi- pative
  hyperbolic systems with a convex entropy.
\newblock {\em Arch. Rational Mech. Anal.}, 169(2):89--117, 2003.

\bibitem{has12}
I.~Hashimoto and Y.~Ueda.
\newblock Asymptotic behaviour of solutions for damped wave equations with
  non-convex convection term on the half line.
\newblock {\em Osaka J. Math.}, 49(1):37--52, 2012.

\bibitem{LHl}
L.~H\"{o}rmander.
\newblock {\em The analysis of linear partial differential operators {III}:
  {P}seudo-differential operators}, volume 274 of {\em Grundlehren der
  mathematischen {W}issenschaften}.
\newblock Springer, Berlin, 1985.

\bibitem{LH}
L.~H\"{o}rmander.
\newblock {\em Lectures on Nonlinear Hyperbolic Differential Equations},
  volume~26 of {\em Mathematiques et Applications}.
\newblock Springer, Berlin, Heidelberg, New York, 1997.

\bibitem{kat17}
M.~Kato and Y.~Ueda.
\newblock Asymptotic profile of solutions for the damped wave equation with a
  nonlinear convection term.
\newblock {\em Math. Methods Appl. Sci.}, 40(18):7760--7779, 2017.

\bibitem{kawa83}
S.~Kawashima.
\newblock {\em Systems of a Hyperbolic-Parabolic Composite Type, with
  Applications to the Equations of Magnetohydrodynamics}.
\newblock PhD thesis, Kyoto University, 1983.

\bibitem{ued07}
S.~Kawashima and Y.~Ueda.
\newblock Large time behavior of solutions to a semilinear hyperbolic system
  with relaxation.
\newblock {\em J. Hyperbolic Differ. Equ.}, 4(1):147--179, 2007.

\bibitem{kawa04}
S.~Kawashima and W.-A. Yong.
\newblock Dissipative structure and entropy for hyperbolic systems of balance
  laws.
\newblock {\em Arch. Rational Mech. Anal.}, 174(3):345--364, 2004.

\bibitem{liu13}
Y.~Liu and S.~Kawashima.
\newblock Global existence and asymptotic decay of solutions to the nonlinear
  {Timoshenko} system with memory.
\newblock {\em Nonlinear Anal.}, 84:1--17, 2013.

\bibitem{met01}
G.~M\'{e}tivier.
\newblock Stability of multidimensional shocks.
\newblock In {\em Advances in the theory of shock waves}, volume~47 of {\em
  Progr. Nonlinear Differential Equations Appl.}, pages 25--103. Birk\"auser,
  Boston, MA, 2001.

\bibitem{mey812}
Y.~Meyer.
\newblock R\'{e}gularit\'{e} des solutions des \'{e}quations aux
  d\'{e}riv\'{e}es partielles non lin\'{e}aires (d'apr\`{e}s {J.-M. B}ony).
\newblock In {\em Bourbaki Seminar{, Vol.} 1979/80}, volume 842 of {\em Lecture
  Notes in Math.}, pages 293--302. Springer, Berlin, New York, 1981.

\bibitem{mey81}
Y.~Meyer.
\newblock Remarques sur un th\'eor\`{e}me de {J.-M.} {B}ony.
\newblock In {\em Proceedings of the Seminar on Harmonic Analysis (Pisa,
  1980)}, number suppl.1 in Rend. Circ. Mat. Palermo (2), pages 1--20. 1981.

\bibitem{Ori06}
R.~Orive and E.~Zuazua.
\newblock Long-time behavior of solutions to a nonlinear hyperbolic relaxation
  system.
\newblock {\em J. Differ. Equ.}, 288(1):17--38, 2006.

\bibitem{rac17}
R.~Racke, W.~Wang, and R.~Xue.
\newblock Optimal decay rates and global existence for a semilinear
  {Timoshenko} system with two damping effects.
\newblock {\em Math. Methods Appl. Sci.}, 40(1):210--222, 2017.

\bibitem{boi97}
T.~Ruggeri and D.~Serre.
\newblock Stability of constant equilibrium state for dissipative balance laws
  system with a convex entropy.
\newblock {\em Quart. Appl. Math.}, 62(1):163--179, 2004.

\bibitem{kawa85}
Y.~Shizuta and S.~Kawashima.
\newblock Systems of equations of hyperbolic-parabolic type with applications
  to the discrete {Boltzmann} equation.
\newblock {\em Hokkaido Math. J.}, 14(2):249--275, 1985.

\bibitem{sro18}
M.~Sroczinski.
\newblock Asymptotic stability of homogeneous states in the relativistic
  dynamics of viscous, heat-conductive fluids.
\newblock {\em Arch. Rational Mech. Anal.}, 231(1):91--113, 2019.

\bibitem{srot19}
M.~Sroczinski.
\newblock {\em Global existence and asymptotic decay for quasilinear
  second-order symmetric hyperbolic systems of partial differential equations
  occuring in the relativistic dynamics of dissipative fluids}.
\newblock PhD thesis, University of Konstanz, 2019.

\bibitem{sro20}
M.~Sroczinski.
\newblock Asymptotic stability in a second-order symmetric hyperbolic system
  modeling the relativistic dynamics of viscous heat-conductive fluids with
  diffusion.
\newblock {\em J. Differ. Equ.}, 268(2):825--851, 2020.

\bibitem{T91}
M.~E. Taylor.
\newblock {\em Pseudodifferential operators and nonlinear {PDE}}, volume 100 of
  {\em Progress in mathematics}.
\newblock Birkh\"auser Boston Inc., Boston, MA, 1991.

\bibitem{ued11}
Y.~Ueda, T.~Nakamura, and S.~Kawashima.
\newblock Energy method in partial {Fourier} space and application to stability
  problems in the half space.
\newblock {\em J. Differ. Equ.}, 250(2):1169--1199, 2015.

\bibitem{whi74}
G.~B. Whitham.
\newblock {\em Linear and nonlinear waves}.
\newblock Pure and Applied Mathematics. Wiley-Interscience [John Wiley \&
  Sons], New York-London-Sydney, 1974.

\bibitem{yon01}
W.-A. Yong.
\newblock Basic aspects of hyperbolic relaxation systems.
\newblock In {\em Advances in the theory of shock waves}, volume~47 of {\em
  Progr. Nonlinear Differential Equations Appl.}, pages 259--305.
  Birkh\"{a}user Boston, Boston, MA, 2001.

\bibitem{yon04}
W.-A. Yong.
\newblock Entropy and global existence for hyperbolic balance laws.
\newblock {\em Arch. Rational Mech. Anal.}, 172(2):247--266, 2004.

\end{thebibliography}

\end{document}